\documentclass[a4paper,twoside,reqno]{amsart}
\usepackage[utf8]{inputenc}

\usepackage{amsmath,amssymb,amsfonts,amsthm,enumerate,mathtools,tikz}

\usepackage{fullpage}
\usepackage{hyperref}
\renewcommand{\MR}[1]{\href{http://www.ams.org/mathscinet-getitem?mr=#1}{MR#1}}

\numberwithin{equation}{section}

\newcommand{\Span}{\mathrm{Span}}
\newcommand{\Nor}{\mathrm{Nor}}

\newcommand{\Alg}{\mathbb{A}}
\newcommand{\cA}{\mathcal{A}}

\newcommand{\cF}{\mathcal{F}}
\newcommand{\cN}{\mathcal{N}}
\newcommand{\cT}{\mathcal{T}}

\newcommand{\cU}{\mathcal{U}}

\newcommand{\Lcal}{{\mathcal{L}}}
\newcommand{\Rcal}{{\mathcal{R}}}

\newcommand{\asX}{\langle X_n \rangle}
\newcommand{\asoX}{\langle X \rangle}

\newcommand{\N}{\mathbb{N}}

\newcommand{\Fcal}{{\mathcal{F}}}

\newcommand{\Ncal}{{\mathcal{N}}}

\def\id{\text{id}}
\theoremstyle{plain}

\newtheorem{thm}{Theorem}[section]
\newtheorem{pro}[thm]{Proposition}
\newtheorem{lem}[thm]{Lemma}
\newtheorem{cor}[thm]{Corollary}

\theoremstyle{definition}

\newtheorem{ex}[thm]{Example}

\newtheorem{facts}[thm]{Facts}
\newtheorem{defnotation}[thm]{Definition-Notation}
\newtheorem{pronotation}[thm]{Proposition-Notation}
\newtheorem{convention}[thm]{Convention}

\newtheorem{dfn}[thm]{Definition}
\newtheorem{que}[thm]{Question}
\newtheorem{rmk}[thm]{Remark}
\newtheorem{problem}[thm]{Problem}

\newcommand{\LM}{\mathbf{LM}}

\title[Segre products of Yang-Baxter algebras]{Segre products and
Segre morphisms in a class of  Yang-Baxter algebras}

\author{Tatiana Gateva-Ivanova}
\address{Max Planck Institute for Mathematics, Vivatsgasse 7, 53111 Bonn, Germany,
and American University in
Bulgaria, 2700 Blagoevgrad, Bulgaria} \email{tatyana@aubg.edu}

\thanks{ Partially supported by the Max Planck Institute for Mathematics, (MPIM),
Bonn, by ICTP, Trieste, and
by Grant KP-06 N 32/1, 07.12.2019 of the Bulgarian National Science Fund.
}

\begin{document}
\date{\today}

\begin{abstract}
Let $(X,r_X)$ and $(Y,r_Y)$ be finite nondegenerate involutive set-theoretic
solutions of  the Yang-Baxter equation, and let $A_X =  \cA(\textbf{k}, X, r_X)$ and $A_Y=  \cA(\textbf{k}, Y, r_Y)$
 be their quadratic Yang-Baxter algebras over a field $\textbf{k}.$
We find an explicit presentation of the Segre product $A_X\circ A_Y$ in terms of one-generators
and quadratic relations. We introduce analogues of Segre maps in the class of Yang-Baxter algebras
and find their images and their kernels.
The results agree with their classical analogues in the commutative case.
   \end{abstract}
\maketitle

\section{Introduction}
\label{Intro}

 It was established in the last three decades that
solutions of the linear braid or Yang-Baxter equations (YBE)
\[r^{12}r^{23}r^{12}=r^{23}r^{12} r^{23} \]
on
a vector space of the form $V^{\otimes 3}$ lead to remarkable
algebraic structures. Here $r : V\otimes V \longrightarrow
V\otimes V,$  $r^{12}= r\otimes \id$, $r^{23} = id\otimes r$ is a
notation and structures include coquasitriangular bialgebras
$A(r)$, their quantum group (Hopf algebra) quotients, quantum
planes and associated objects, at least in the case of specific
standard solutions, see \cite{MajidQG, RTF}. On the other hand,
the variety of all solutions on vector spaces of a given dimension
has remained rather elusive in any degree of generality. It was
proposed by V.G. Drinfeld \cite{D}, to consider the same equations
in the category of sets, and in this setting numerous  results
were found. It is clear that a set-theoretic solution extends to a
linear one, but more important than this is that set-theoretic
solutions lead to their own remarkable algebraic and combinatoric
structures, only somewhat analogous to quantum group
constructions. In the present paper we continue our systematic
study of set-theoretic solutions based on the associated quadratic
algebras that they generate.

In this paper "a solution of YBE" , or shortly, "a solution" means "a
nondegenerate involutive set-theoretic solution of YBE", see Definition
\ref{def:quadraticsets_All}.

The  Yang-Baxter algebras $\cA_X= \cA(\textbf{k}, X, r)$ related to solutions $(X,
r)$ of finite order $n$
 will play a central role in the paper.
It was proven in \cite{GIVB} and \cite{GI12}  that
these are quadratic algebras with remarkable algebraic, homological and
combinatorial properties: they are noncommutative, but preserve the good
properties of the commutative polynomial rings $\textbf{k} [x_1, \cdots , x_n]$.
Each such an algebra $\cA$ has finite global dimension and polynomial growth, it is
Koszul and a Noetherian domain. In the special case when $(X,r)$ is square-free,
$\cA$ is an Artin-Schelter regular PBW algebra (has a basis of Poincar\'{e}-Birkhoff-Witt type)  with
respect to some enumeration $X= \{x_1, \dots, x_n\}$ of $X$.
More
precisely, $\cA$ is a binomial skew polynomial ring in the sense of \cite{GI96}
which implies its good combinatorial and computational properties (the use of
noncommutative Gr\"{o}bner bases). Conversely, every binomial skew polynomial ring in the sense of \cite{GI96}
defines via its quadratic relations  a square-free solutions $(X, r)$ of YBE.
It seems interesting to find more
analogues coming from commutative algebra and algebraic geometry.
The  algebras $\cA_X= \cA(\textbf{k}, X, r)$  associated to  multipermutation
(square-free) solutions of
level two were studied in  \cite{GIM11}, we referred to them as 'quantum spaces'.
In this special case
a first stage of noncommutative geometry on $\cA_X= \cA(\textbf{k}, X, r)$  was
proposed, see  \cite{GIM11}, Section 6.
In \cite{AGG} a class of particular quadratic PBW
algebras called "noncommutative projective spaces" were studied
and analogues of  Veronese and Segre morphisms between noncommutative projective spaces were
introduced and studied.
In \cite{GI22} we studied the Veronese subalgebras
of Yang-Baxter algebras $\cA= \cA(\textbf{k}, X, r)$ related to finite solutions $(X,
r)$ and introduced noncommutative analogues of the Veronese maps in the class of Yang-Baxter algebras of finite solutions.

In the present paper we consider the following problem.
\begin{problem}
\label{problem}
Let $(X,r_X)$ and $(Y, r_Y)$ be finite nondegenerate symmetric sets whose Yang-Baxter algebras are $A = \cA(\textbf{k}, X, r)$ and $B = \cA(\textbf{k}, Y, r_Y),$
respectively.
\begin{enumerate}
\item
Find a presentation of the Segre product  $A\circ B$
 in terms of one-generators and linearly independent quadratic relations.
\item Introduce analogues of Segre maps for the class
of Yang-Baxter algebras of finite nondegenerate symmetric sets.
\item Study the special case when $(X,r_X)$ and $(Y, r_Y)$ are square-free solutions.
\end{enumerate}
\end{problem}

Our main results are Theorem \ref{thm:rel_segre}, Theorem  \ref{thm:segremap} and Theorem  \ref{thm:segre_square-free} which solve completely the problem.

The paper is organized as follows. In Section 2 we recall basic definitions and
facts used throughout the paper.
In Section \ref{sec:YBalgebras} we consider the quadratic algebra
$\cA(\textbf{k},  X, r)$ of a finite nondegenerate
symmetric set $(X,r)$.
 We fix the main settings and conventions and collect some of the most important
properties of the Yang-Baxter algebras
$\cA_X= \cA(\textbf{k}, X, r)$  used in the paper.
In Section \ref{Sec:SegreProduct} we study the Segre product $A\circ B$ of the Yang-Baxter algebras
$A=\cA(\textbf{k}, X, r_1)$ and $B=\cA(\textbf{k}, Y, r_2)$ of two finite solutions $(X, r_1)$ and $(Y, r_2)$, respectively.  We prove Theorem  \ref{thm:rel_segre} which gives an explicit finite presentation  of  the Segre product $A\circ B$
in terms of one-generators and linearly independent quadratic relations.
In Section \ref{Sec:SegreMaps}
we introduce analogues of Segre morphisms $s_{m,n}$ for quantum spaces $A=\cA(\textbf{k}, X, r_1)$ and $B=\cA(\textbf{k}, Y, r_2)$ related to
finite solutions $(X,r_1)$ and $(Y,r_2)$ of orders $m$ and $n$, respectively. We involve an abstract solution
$(Z,r_Z)$ of order $mn$ which is isomorphic to the Cartesian product of solutions $(X\times Y, r_{X\times Y})$ and define
the Segre map $s_{m,n}\;  \cA(\textbf{k}, Z, r_Z) \longrightarrow A \otimes B$.
Theorem
\ref{thm:segremap} shows that the image of the  map $s_{m,n}$ is the Segre product $A\circ B$ and describes explicitly a minimal set of
generators for its kernel.
Corollary \ref{cor:SegreProdNoetherian} shows that the Segre product $A\circ B$ is left and right Noetherian.
The results agree with their classical analogues in the commutative case, \cite{Harris}.
We end the section with an open question, see Question \ref{que}.
In Section \ref{sec:square-free} we pay special attention to the class of square-free solutions. In the case when $(X, r_1)$ and $(Y, r_2)$
are finite square-free solutions, the algebras $A=\cA(\textbf{k}, X, r_1)$ and $B=\cA(\textbf{k}, Y, r_2)$  are binomial skew polynomial algebras with respect to proper enumerations of $X$ and $Y$.
Theorem  \ref{thm:segre_square-free} shows that in this case the Segre product  $A\circ B$ is a PBW algebra and gives an explicit standard finite presentation in terms of PBW generators and quadratic relations which form a (noncommutative) Gr\"{o}bner basis.
Moreover, the Segre maps are well defined in the subclass of YB-algebras related to square-free solutions.
In Section \ref{sec:examples} we give an example which illustrates our results.

\section{Preliminaries}
\label{seq:preliminaries}
Let $X$ be a non-empty set, and let $\textbf{k}$ be a field.
We denote by $\asoX$
the free monoid
generated by $X,$ where the unit is the empty word denoted by $1$, and by
$\textbf{k}\asoX$-the unital free associative $\textbf{k}$-algebra
generated by $X$. For a non-empty set $F
\subseteq \textbf{k}\asoX$, $(F)$ denotes the two sided ideal
of $\textbf{k}\asoX$ generated by $F$.
When the set $X$ is finite, with $|X|=n$,  and ordered, we write $X= \{x_1,
\dots, x_n\}$
and fix the degree-lexicographic order $<$ on $\asoX$, where $x_1< \dots
<x_n$. As usual, $\N$ denotes the set of all positive integers,
and $\N_0$ is the set of all non-negative integers.

We shall consider associative graded $\textbf{k}$-algebras.
Suppose $A= \bigoplus_{m\in\N_0}  A_m$ is a graded $\textbf{k}$-algebra
such that $A_0 =\textbf{k}$, $A_pA_q \subseteq A_{p+q}, p, q \in \N_0$, and such
that $A$ is finitely generated by elements of positive degree. Recall that its
Hilbert function is $h_A(m)=\dim A_m$
and its Hilbert series is the formal series $H_A(t) =\sum_{m\in\N_0}h_{A}(m)t^m$.
In particular, the algebra $\textbf{k} [X]=\textbf{k} [x_1, \cdots, x_n]$ of commutative polynomials satisfies
\begin{equation}
\label{eq:hilbert}
h_{\textbf{k} [X]}(d)= \binom{n+d-1}{d}= \binom{n+d-1}{n-1} \quad\mbox{and}\quad
H_{\textbf{k} [X]}= \frac{1}{(1 -t)^{n}}.
\end{equation}
We shall use the \emph{natural grading by length} on the free associative algebra
$\textbf{k}\asoX$.
For $m \geq 1$,  $X^m$ will denote the set of all words of length $m$ in $\asoX$,
where the length of $u = x_{i_1}\cdots x_{i_m} \in X^m$
will be denoted by $|u|= m$.
Then
\[\asoX = \bigsqcup_{m\in\N_0}  X^{m},\;
X^0 = \{1\},\;  \mbox{and} \ \   X^{k}X^{m} \subseteq X^{k+m},\]
so the free monoid $\asoX$ is naturally \emph{graded by length}.

Similarly, the free associative algebra $\textbf{k}\asoX$ is also graded by
length:
\[\textbf{k}\asoX
 = \bigoplus_{m\in\N_0} \textbf{k}\asoX_m,\quad \mbox{ where}\ \
 \textbf{k}\asoX_m=\textbf{k} X^{m}. \]

A polynomial $f\in  \textbf{k}\asoX$ is \emph{homogeneous of degree $m$} if $f \in
\textbf{k} X^{m}$.
We denote by
\[\cT =\cT(X) :=\left\lbrace x_1^{\alpha_1}\cdots x_n^{\alpha_n}\in\asoX \ \vert
\ \alpha_i\in\N_0, i\in\{1,\dots,n\}\right\rbrace\]
the set of ordered monomials (terms) in $\asoX$.

\subsection{Gr\"obner bases for ideals in the free associative algebra}
\label{sec:grobner}
We shall briefly remind some basics on noncommutative Gr\"obner bases.
In this subsection $X=\{x_1,\dotsc, x_n\}$.
Suppose $f \in \textbf{ k}\asoX$ is a nonzero polynomial. Its leading
monomial with respect to the degree-lexicographic order $<$ on $\asoX$ will be denoted by
$\LM(f)$.
One has $\LM(f)= u$ if
$f = cu + \sum_{1 \leq i\leq m} c_i u_i$, where
$ c,c_i \in \textbf{k}$, $c \neq 0 $ and $u > u_i$ in $\asoX$, for every
$i\in\{1,\dots,m\}$.
Given a set $F \subseteq \textbf{k} \asoX$ of
non-commutative polynomials, $\LM(F)$ denotes the set
 \[\LM(F) = \{\LM(f) \mid f \in F\}.\]
A monomial $u\in \asoX$ is \emph{normal modulo $F$} if it does not contain any of
the monomials $\LM(f), f \in F$ as a subword.
 The set of all normal monomials modulo $F$ is denoted by $N(F)$.

Let  $I$ be a two sided graded ideal in $\textbf{ k} \asoX$ and let $I_m = I\cap
\textbf{k}X^m$.
We shall
assume that
$I$ \emph{is generated by homogeneous polynomials of degree $\geq 2$}
and $I = \bigoplus_{m\ge 2}I_m$. Then the quotient
algebra $A = \textbf{k} \asoX/ I$ is finitely generated and inherits its
grading $A=\bigoplus_{m\in\N_0}A_m$ from $ \textbf{k} \asoX$. We shall work with
the so-called \emph{normal} $\textbf{k}$-\emph{basis of} $A$.
We say that a monomial $u \in \asoX$ is  \emph{normal modulo $I$} if it is normal
modulo $\LM(I)$. We set
\[N(I):=N(\LM(I)).\]
In particular, the free
monoid $\asoX$ splits as a disjoint union
\begin{equation}
\label{eq:X1eq2a}
\asoX=  N(I)\sqcup \LM(I).
\end{equation}
The free associative algebra $\textbf{k} \asoX$ splits as a direct sum of
$\textbf{k}$-vector
  subspaces
  \[\textbf{k} \asoX \simeq  \Span_{\textbf{k}} N(I)\oplus I,\]
and there is an isomorphism of vector spaces
$A \simeq \Span_{\textbf{k}} N(I).$

It follows that every $f \in \textbf{k}\asoX$ can be written uniquely as $f =
h+f_0,$ where $h \in I$ and $f_0\in {\textbf{k}} N(I)$.
The element $f_0$ is called \emph{the normal form of $f$ (modulo $I$)} and denoted
by
$\Nor(f)$.
We define
\[N(I)_{m}=\{u\in N(I)\mid u\mbox{ has length } m\}.\]
In particular, $N(I)_{1} = X$, and by definition $N(I)_{0} = 1.$
 Then
$A_m \simeq \Span_{\textbf{k}} N(I)_{m}$ for every $m\in\N_0$.

A subset
$G \subseteq I$
of monic polynomials is a \emph{Gr\"{o}bner
basis} of $I$ (with respect to the ordering $<$) if
\begin{enumerate}
\item $G$ generates $I$ as a
two-sided ideal, and
\item for every $f \in I$ there exists $g \in G$ such that $\LM(g)$ is a
    subword of $\LM(f)$, that is
$\LM(f) = a\LM(g)b$,  for some $a, b \in \asoX$.
\end{enumerate}
A  Gr\"{o}bner basis $G$ of
$I$ is \emph{reduced} if (i)  the set $G\setminus\{f\}$ is not a Gr\"{o}bner
basis of $I$, whenever $f \in G$; (ii) each  $f \in G$  is a linear combination
of normal monomials modulo $G\setminus\{f\}$.

It is well-known that every ideal $I$ of $\textbf{k} \asoX$ has a unique reduced
Gr\"{o}bner basis $G_0= G_0(I)$ with respect to $<$. However, $G_0$ may be
infinite.
 For more details, we refer the reader to \cite{Latyshev, Mo88, Mo94}.

Bergman's Diamond lemma  \cite[Theorem 1.2]{Bergman} implies the following.
\begin{rmk}
\label{rmk:diamondlemma}
Let $G  \subset \textbf{k}\asoX$  be a set  of noncommutative polynomials. Let $I =
(G)$ and let $A = \textbf{k}\asoX/I.$ Then the following
conditions are equivalent.
\begin{enumerate}
\item
The set
$G$  is a Gr\"{o}bner basis of $I$.

\item Every element $f\in \textbf{k}\asX$ has a unique normal form modulo $G$,
    denoted by $\Nor(f)$.
\item
There is an equality $N(G) = N(I)$, so there is an isomorphism of vector spaces
\[\textbf{k}\asoX \simeq I \oplus \textbf{k}N(G).\]
\item The image of $N(G)$ in $A$ is a $\textbf{k}$-basis of $A$.
In this case
$A$ can be identified with the $\textbf{k}$-vector space $\textbf{k}N(G)$, made
a $\textbf{k}$-algebra by the multiplication
$a\bullet b: = \Nor(ab).$
\end{enumerate}
\end{rmk}

 In this paper, we focus on a class of quadratic finitely presented algebras $A$
 associated with finite set-theoretic solutions $(X,r)$  of the Yang-Baxter equation.
 Following Yuri Manin, \cite{Ma88}, we call them Yang-Baxter algebras.

\subsection{Quadratic algebras}
\label{sec:Quadraticalgebras}
A quadratic  algebra is an associative graded algebra
 $A=\bigoplus_{i\ge 0}A_i$ over a ground field
 $\textbf{k}$  determined by a vector space of generators $V = A_1$ and a
subspace of homogeneous quadratic relations $R= R(A) \subset V
\otimes V.$ We assume that $A$ is finitely generated, so $\dim A_1 <
\infty$. Thus $ A=T(V)/( R)$ inherits its grading from the tensor
algebra $T(V)$.

%Following the classical tradition
%(and a recent trend),
As usual, we take a
combinatorial approach to study $A$. The properties of $A$ will be
read off a finite presentation $A= \textbf{k} \langle X\rangle /(\Re)$,
where by convention $X$ is a fixed finite set of generators of
degree $1$, ($X$ is a basis of $A_1$), $|X|=n,$
and $(\Re)$ is the two-sided
ideal of relations, generated by a finite linearly independent set $\Re$ of
homogeneous polynomials of degree two.
 \begin{dfn}
\label{def:PBW}
A quadratic algebra $A$ is
\emph{a  Poincar\`{e}–Birkhoff–Witt type algebra} or shortly
\emph{a PBW algebra} if there exists an enumeration $X= \{x_1,
\cdots, x_n\}$ of $X,$ such that the quadratic relations $\Re$ form a
(noncommutative) Gr\"{o}bner basis with respect to the
degree-lexicographic ordering $<$ on $\asoX$.
In this case the set of normal monomials
(mod $\Re$) forms a $\textbf{k}$-basis of $A$ called a \emph{PBW
basis}
 and $x_1,\cdots, x_n$ (taken exactly with this enumeration) are called \emph{
 PBW-generators of $A$}.
\end{dfn}
 The notion of a \emph{PBW} algebra was introduced by Priddy, \cite{priddy}.
 His \emph{PBW basis}  is a generalization of the classical
 Poincar\'{e}-Birkhoff-Witt basis for the universal enveloping of a  finite
 dimensional Lie algebra.
 PBW algebras form an important class of Koszul algebras.
  The interested reader can find information on quadratic algebras and, in
  particular, on Koszul algebras and PBW algebras in
  \cite{PoPo}.
A special class of PBW algebras important for this paper, are the
\emph{binomial skew polynomial rings}.

%The binomial skew polynomial rings were introduced by the author in \cite{GI96}, initially they were called "skew polynomial rings with binomial %relations".  They form a class of quadratic PBW algebras with
%remarkable properties: they are noncommutative, but preserve the good algebraic
%and homological properties of the commutative polynomial rings $\textbf{k} [x_1,
%\cdots , x_n]$, each such an algebra $A$ is a Noetherian Artin-Schelter regular
%domain, it is Koszul, and at the same time it

  \begin{dfn}
\label{binomialringdef}
\cite{GI96, GI94} A {\em binomial skew polynomial ring} is a quadratic algebra
 $A=\textbf{k} \langle x_1, \cdots , x_n\rangle/(\Re_0)$ with
precisely $\binom{n}{2}$ defining relations
\begin{equation}
\label{eq:skewpol}
\Re_0=\{f_{ji} = x_{j}x_{i} -
c_{ij}x_{i^\prime}x_{j^\prime} \mid 1\leq i<j\leq n\}\quad\text{such that}
\end{equation}

\begin{enumerate}
\item[(a)]
 $c_{ij} \in \textbf{k}^{\times}$;
\item[(b)] For every pair $i, j, \; 1\leq
i<j\leq n$, the relation $x_{j}x_{i} - c_{ij}x_{i'}x_{j'}\in \Re_0,$
satisfies $j
> i^{\prime}$, $i^{\prime} < j^{\prime}$;
\item[(c)] Every ordered monomial $x_ix_j,$
with $1 \leq i < j \leq n$ occurs (as a second term) in some
relation in $\Re_0$;
\item[(d)] The set $\Re_0$ is the
{\it reduced Gr\"obner basis } of the two-sided ideal $(\Re_0)$,
with respect to the degree-lexicographic order $<$ on $\asoX$,  or
equivalently
\item[(d$^{\prime}$)] The set of terms $\cT=\left\lbrace x_1^{\alpha_1}\cdots x_n^{\alpha_n}\in\asoX \ \vert
\ \alpha_i\in\N_0, i\in\{0,\dots,n\}\right\rbrace$ projects to a $\textbf{k}$-basis of $A$.
\end{enumerate}
\end{dfn}
The equivalence of (d) and (d$^{\prime}$) follows from Remark \ref{rmk:diamondlemma}.

Clearly, $A$ is a PBW algebra with a set of PBW generators $x_1, \cdots , x_n$.
Each binomial skew polynomial ring defines via its relations a square-free
solution of the Yang-Baxter equation, see \cite{GIVB}. Conversely, if $(X,r)$ is a square-free solution, then there exists an enumeration
$X=\{x_1, x_2, \cdots, x_n\}$ such that the Yang-Baxter algebra $\cA(\textbf{k}, X,r)$ is a binomial skew-polynomial ring, see \cite{GI12},

\begin{ex}
\label{example1}
Let $A = \textbf{k}\langle x_1,x_2,x_3,x_4 \rangle /(\Re_0)$, where
\[
 \begin{array}{l}
\Re_0
 = \{x_4x_2 -x_1x_3,\; x_4x_1 -x_2x_3,\;x_3x_2 -x_1x_4,\; x_3x_1 -x_2x_4, \;
 x_4x_3 -x_3x_4,\;  x_2x_1 -x_1x_2
\}.
\end{array}
\]
The algebra $A$ is a binomial skew-polynomial ring. It is a PBW algebra with PBW
generators $X= \{x_1, x_2, x_3, x_4\}$.
The relations of $A$ define in a natural way a solution of the Yang-Baxter equation.
 \end{ex}

\subsection{Set-theoretic solutions of the Yang-Baxter equation and their Yang-Baxter algebras}

\begin{dfn}
\label{def:quadraticsets_All}
Let $X $ be a nonempty set, and let $r: X\times X\longrightarrow X\times X$ be a bijective map.
Then the pair $(X, r)$ is called \emph{a quadratic set}.
Recall that the map $r$ is \emph{a set-theoretic solution of the Yang–Baxter equation (YBE)}
if  the braid
relation
\[r^{12}r^{23}r^{12} = r^{23}r^{12}r^{23}\]
holds in $X\times X\times X,$  where  $r^{12} = r\times\id_X$, and
$r^{23}=\id_X\times r$. In this case we  refer to  $(X,r)$ also as
\emph{a braided set}.

The image of $(x,y)$ under $r$ is
presented as
\[
r(x,y)=({}^xy,x^{y}).
\]
This formula defines a ``left action'' $\Lcal: X\times X
\longrightarrow X,$ and a ``right action'' $\Rcal: X\times X
\longrightarrow X,$ on $X$ as: $\Lcal_x(y)={}^xy$, $\Rcal_y(x)=
x^{y}$, for all $x, y \in X$.
\begin{enumerate}
\item[(i)] $(X, r)$ is \emph{non-degenerate}, if
the maps $\Lcal_x$ and $\Rcal_x$ are bijective for each $x\in X$.
\item[(ii)]
$(X, r)$ is \emph{involutive} if $r^2 = id_{X\times X}$.
\item[(iii)]
$(X,r)$ is \emph{square-free} if $r(x,x)=(x,x)$ for all $x\in X.$
\item[(iv)]  A braided set $(X,r)$ with $r$ involutive is
called \emph{a symmetric set}.
\item[(v)] A nondegenerate symmetric set will be called simply \emph{a solution}.

\end{enumerate}
\end{dfn}

\begin{convention}
\label{conv:convention1} In this paper we shall always assume that $(X, r)$ is
nondegenerate.   "\emph{A solution}"
means "\emph{a non-degenerate symmetric set}" $(X,r)$, where $X$  is
a set of arbitrary cardinality.
\end{convention}
As a notational tool, we  shall  often identify the sets $X^{\times
m}$ of ordered $m$-tuples, $m \geq 2,$  and $X^m,$ the set of all
monomials of length $m$ in the free monoid $\asX$. Sometimes for simplicity we
shall  write $r(xy)$ instead of $r(x,y)$.

\begin{dfn} \cite{GI04, GIM08}
\label{def:algobjects} To each quadratic set $(X,r)$ we associate
canonically algebraic objects generated by $X$ and with quadratic
relations $\Re =\Re(r)$ naturally determined as
\[
xy=y^{\prime} x^{\prime}\in \Re(r)\; \text{iff}\;
 r(x,y) = (y^{\prime}, x^{\prime})\; \text{and} \;
 (x,y) \neq (y^{\prime}, x^{\prime})\;\text{hold in}\;X \times X.
\]
In this paper we use mainly the monoid and the quadratic algebra associated with $(X,r)$.

 The monoid
$S =S(X, r) = \langle X ; \; \Re(r) \rangle$
 with a set of generators $X$ and a set of defining relations $ \Re(r)$ is
called \emph{the monoid associated with $(X, r)$}.
 %The \emph{group $G=G(X, r)= G_X$ associated with} $(X, r)$ is
%defined analogously.

For an arbitrary fixed field $\textbf{k}$,
\emph{the} \textbf{k}-\emph{algebra associated with} $(X ,r)$ is
defined as
\[\begin{array}{c}
\cA = \cA(\textbf{k},X,r) = \textbf{k}\langle X  \rangle /(\Re_0),
\text{where}\;
\Re_0 = \{xy-y^{\prime}x^{\prime}\mid r(xy)=y^{\prime}x^{\prime}\; \text{and}\; r(xy)\neq xy \; \text{holds in}\;X^2 \}.
\end{array}
\]
Clearly, $\cA$ is a quadratic algebra generated by $X$ and
 with defining relations  $\Re_0 $ (or equivalently, $\Re(r)$ ), which is
isomorphic to the monoid algebra $\textbf{k}S(X, r)$.
  When $(X,r)$ is a solution, following Yuri Manin, \cite{Ma88}, we call
the algebra $\cA$   \emph{an Yang-Baxter algebra}, or shortly \emph{YB algebra}.
\end{dfn}

Suppose $(X,r)$ is a finite quadratic set. Then
$A=  A(\textbf{k},X,r)$ is \emph{a connected
graded} $\textbf{k}$-algebra (naturally graded by length),
 $A=\bigoplus_{i\ge0}A_i$, where
$A_0=\textbf{k},$
and each graded component $A_i$ is finite dimensional.

%Moreover, the associated monoid $S= S(X,r)$ \emph{is naturally graded by
%length}:
%\begin{equation}
%\label{eq:Sgraded}
%S = \bigsqcup_{i \geq 0}  S_{i}, \;\; \text{where}\;\;
%S_0 = 1,\; S_1 = X,\; S_i = \{u \in S \mid\;   |u|= i \}, \; S_i.S_j \subseteq
%S_{i+j}.
%\end{equation}
%In the sequel, by "\emph{a graded monoid} $S$", we shall
%mean that $S$ is generated  by $S_1=X$ and graded by length.
%The grading of $S$ induces a canonical grading of its monoid algebra
%$\textbf{k}S(X, r).$
%The isomorphism $A\cong \textbf{k}S(X, r)$ agrees with the canonical gradings,
%so there is an isomorphism of vector spaces $A_m \cong Span_{\textbf{k}}S_m$.

 By \cite[Proposition 2.3.]{GI11} If
$(X,r)$ is a nondegenerate  involutive quadratic set of finite order $|X| =n$
then the set $\Re(r)$
consists of precisely $\binom{n}{2}$ quadratic relations.
In this  case the associated algebra $\cA= \cA(\textbf{k}, X, r)$ satisfies
\[\dim \cA_2 = \binom{n+1}{2}.\]

\begin{defnotation}\cite{GI21}
\label{def:fixedpts}
Suppose $(X,r)$ is an involutive quadratic set. Then the cyclic group $\langle r \rangle =\{1, r\}$  acts on the set $X^2$ and  splits it into disjoint $r$-orbits $\{ xy, r(xy)\}$, where $xy\in X^2$.

An $r$-orbit $\{ xy, r(xy)\}$
is \emph{non-trivial} if $xy\neq r(xy)$.

The element $xy\in X^2$ is \emph{an
$r$-fixed point} if $r(xy) =xy$. \emph{The set of $r$-fixed points} in $X^2$
will be denoted by $\Fcal (X,r)$:
\begin{equation}
\label{eq:fixedpts}
\Fcal (X,r) = \{xy\in X^2\mid r(xy) = xy\}
\end{equation}
\end{defnotation}

The following corollary is a consequence of \cite[Lemma 3.7]{GI21}.
\begin{cor}
\label{cor:orbits_and_fixedpoints}
Let  $(X,r)$ be a nondegenerate symmetric set of finite order $|X|= n$, $\cA= \cA(\textbf{k}, X, r)$.
\begin{enumerate}
   \item
There are exactly $n$ fixed points
$\Fcal = \Fcal (X,r) =  \{x_1y_1, \cdots , x_ny_n\}\subset X^2$ ,
so
$|\Fcal(X,r)|= |X|= n.$
In the special case, when $(X,r)$ is a square-free solution, one has
$\Fcal(X,r) = \Delta_2=\{xx\mid x\in X\}$, the diagonal of $X^2$.
\item The number of non-trivial $r$-orbits is exactly $\binom{n}{2}$. Each such an orbit has two elements:
$xy$ and $r(xy)$, where $xy, r(xy) \in X^2$.
\item
The set $X^2$  splits into $\binom{n+1}{2}$ $r$-orbits.
For $xy,zt\in X^2$ there is an equality $xy=zt$ in $\cA$ \emph{iff} $zt =r(xy)$.
\end{enumerate}
\end{cor}

\section{The quadratic algebra $\cA(\textbf{k}, X, r)$  of a finite
nondegenerate symmetric set $(X,r)$}
\label{sec:YBalgebras}

The following results on square-free solutions are extracted from works of the author.
\begin{facts}
\label{fact1}
\begin{enumerate}
\item \cite[Theorem 3.7]{GI12}
Suppose $(X,r)$ is a square-free nondegenerate and involutive quadratic set of order $n$.
Let and  $\cA = \cA(\textbf{k}, X, r)$  be the associated quadratic algebra.
The following conditions are equivalent.
\begin{enumerate}
\item
 $(X,r)$ is a solution of YBE.
 \item $\cA$ is an Artin-Schelter regular PBW algebra.
 \item
There exists an enumeration
$X=\{x_1, x_2, \cdots, x_n\}$ such that $\cA$ is a binomial skew-polynomial algebra, see Definition \ref{binomialringdef}.
\end{enumerate}
\item \cite[Theorem 3.8]{GI22}
Suppose $(X,r)$ is a nondegenerate symmetric set of order $n$, and $\cA=\cA(K,  X, r)$ is its Yang-Baxter algebra.
Then $\cA$ is a PBW algebra with a set of PBW generators $X = \{x_1, x_2, \cdots, x_n\}$  (enumerated properly)
  if and only if $(X,r)$ is a square-free solution.
In this case $\cA$ is a binomial skew-polynomial algebra (in the sense of \cite{GI96}).
\end{enumerate}
\end{facts}

The following conventions will be kept in the sequel.
\begin{convention}
\label{rmk:conventionpreliminary1}  Let $(X,r)$ be a finite nondegenerate
symmetric set of order $n$, and let $\cA = \cA(\textbf{k},X,r)$ be the associated
Yang-Baxter algebra.
(a) If $(X,r)$ is square-free we fix an
enumeration such that $X= \{x_1, \cdots, x_n\}$ is a set of PBW generators
of $\cA$. In this case $\cA$ is a binomial skew polynomial ring, see
 Definition \ref{binomialringdef}.
% (\ref{eq:Algebra}).
 (b)  If $(X,r)$ is not square-free we fix an arbitrary enumeration
$X=\{x_1, \cdots, x_n\}$ on $X$.
In each case we extend the fixed enumeration on $X$ to the
 degree-lexicographic ordering $<$ on
$\langle X \rangle$.
By convention the Yang-Baxter algebra $\cA =\cA_X= \cA(\textbf{k},X, r)$
is presented as
\begin{equation}
\label{eq:Algebra}
\begin{array}{c}
\cA =\c A(\textbf{k},X,r) = \textbf{k}\langle X  \rangle /(\Re)
\simeq \textbf{k}\langle X ; \;\Re(r)
\rangle ,\;\;\text{where}\\
\Re = \Re_{\cA}
 = \{xy-y^{\prime}x^{\prime}\mid xy> y^{\prime}x^{\prime},  \; \text{and}\; r(xy)=y^{\prime}x^{\prime}
 \}.
\end{array}
\end{equation}
Consider the two-sided ideal  $I=(\Re)$ of  $\textbf{k}\langle X  \rangle$,
let $G= G(I)$ be the  unique reduced
Gr\"{o}bner basis of $I$  with respect to $<$.
It follows from the shape of the relations  $\Re$ that $\Re\subseteq G$ and $G$ is finite, or countably
 infinite, and consists of homogeneous binomials $f_j= u_j-v_j,$ with
 $\LM(f_j)= u_j > v_j$, and $|u_j|= |v_j|$.

The set of all normal monomials modulo $I$ is denoted by $\cN$. As we mentioned
in Section 2, $\cN= \cN(I) = \cN(G)$.
An element $f \in \textbf{k} \asoX $ is in normal form (modulo $I$), if $f \in
\Span_{\textbf{k}} \cN $.
The free
monoid $\asoX$ splits as a disjoint union
$\asoX=  \cN\sqcup \LM(I).$
The free associative algebra $\textbf{k} \asoX$ splits as a direct sum of
$\textbf{k}$-vector
  subspaces
  $\textbf{k} \asoX \simeq  \Span_{\textbf{k}} \cN \oplus I$,
and there is an isomorphism of vector spaces
$\cA \simeq  \Span_{\textbf{k}} \cN.$
We define
\begin{equation}
\label{eq:Nm}
\cN_{m}=\{u\in \cN \mid u\mbox{ has length } m\}.
\end{equation}
 Then
$\cA_m \simeq \Span_{\textbf{k}} \cN_{m}$ for every $m\in\N_0$.
In particular $\dim \cA_m  = |\cN_{m}|, \; \forall m \geq 0$.

Note that since the set of relations $\Re$ is a finite set of homogeneous polynomials,
the elements
of  the reduced Gr\"{o}bner basis $G = G(I)$  of degree $\leq m$ can be found
effectively, (using the standard strategy for constructing a Gr\"{o}bner
basis) and therefore the set of normal monomials $\cN_{m}$  can be found
inductively for $m=1, 2, 3, \cdots .$
Here we do not need an explicit description of the reduced Gr\"{o}bner basis
$G$  of $I$.

Let $\cN$ be the set of normal monomials
modulo the ideal $I = (\Re)$.
It follows from Bergman's Diamond lemma, \cite[Theorem 1.2]{Bergman}, that
if we consider the space $\textbf{k} \cN$
endowed with multiplication defined by
 \[f \bullet g := \Nor (fg), \quad \text{for every}\; f,g \in \textbf{k} \cN\]
then
$(\textbf{k} \cN, \bullet )$
has a well-defined structure of a graded algebra, and there is an isomorphism of
graded algebras
\[
\cA=\cA(\textbf{k}, X, r) \cong (\textbf{k} \cN, \bullet ),\;\; \text{so}\; \; \cA =\bigoplus_{m\in\N_0}  A_m \cong \bigoplus_{m\in\N_0}  \textbf{k}\cN_m.
\]
We shall often identify the algebra  $\cA$ with $(\textbf{k}\cN,
\bullet )$.
%Similarly, we consider an operation $\bullet$ on the set $\cN$, with $a
%\bullet b:= \Nor(ab), \quad \text{for every}\; a,b \in \cN$
%and identify the monoid $S=S(X,r)$ with $(\cN, \bullet)$, see \cite{Bergman},
%Section 6.

In the case when $(X,r)$ is square-free,  the set of normal monomials is exactly
$\cT$, so $\cA$ is identified with $(\textbf{k} \cT, \bullet )$ and
$S(X,r)$ is identified with $(\cT, \bullet)$.
\end{convention}

%It was proven through the years that the Yang-Baxter algebras $\cA(\textbf{k},
%X, r)$ coresponding to finite nondegenerate symmetric sets have remarkable algebraic and
%homological properties. They are noncommutative, but have many of the "good"
%properties of the commutative polynomial ring
%$\textbf{k}[x_1, \cdots, x_n]$.

The following results are extracted from \cite{GIVB}.
\begin{facts}
\label{fact2}
Suppose $(X,r)$ is a nondegenerate symmetric set of order $n$, $X = \{x_1,
\cdots , x_n\}$, let $S=S(X,r)$ be the associated monoid and  and let $\cA = \cA(\textbf{k},
X, r)$  the associated Yang-Baxter algebra ($\cA$ is isomorphic to the monoid
algebra $\textbf{k}S$).
Then the following conditions hold.
\begin{enumerate}
\item
$S$ is a semigroup of $I$-type, that is there is a bijective map $v: \cU
\mapsto S$, where $\cU$ is the free $n$-generated abelian monoid
$\cU=[u_1, \cdots, u_n]$ such that $v(1) = 1,$ and such that
\[
\{v(u_1a), \cdots, v(u_na)\} = \{ x_1v(a), \cdots, x_nv(a) \}, \;\text{for
all}\; a \in \cU.
\]
\item The Hilbert series of $A$ is $H_A(t)= 1/(1-t)^n.$
\item \cite[Theorem 1.4]{GIVB}
(a) $A$ has finite global dimension and polynomial growth; (b) $A$ is Koszul;
 (c) A is left and right Noetherian; (d) $A$ satisfies the Auslander condition and is
 is Cohen-Macaulay,
(e) $A$ is finite over its center.

\item \cite[Corollary 1.5]{GIVB}
$A$ is a domain, and in particular the monoid $S$ is cancellative.
\end{enumerate}
\end{facts}
Note that  (1) is a consequence of Theorem 1.3 in \cite{GIVB}, see more details in \cite{GI22}.
Part (2) is straightforward from (1).

\begin{cor}
\label{cor:dimAd}
In notation and conventions as above. Let $(X,r)$ be a nondegenerate symmetric
set of order $n$. Then for every integer $d \geq 1$ there are equalities
\begin{equation}
\label{eq:dimAd}
\dim \cA_d=\binom{n+d-1}{d} = |\cN_d|.
\end{equation}
\end{cor}

Important properties of the square-free solutions are given in \cite[Theorem 1.2]{GI12}.
\section{Segre products of Yang-Baxter algebras}
\label{Sec:SegreProduct}
In this section we investigate the Segre product of Yang-Baxter algebras. The main result of the section is Theorem
\ref{thm:rel_segre}.
\subsection{Segre products of quadratic algebras}
In \cite{Froberg} Fr\"{o}berg and Backelin  made a systematic account on Koszul algebras and showed  that their properties are preserved under various constructions such as tensor products, Segre products, Veronese subalgebras.
Our main reference on Segre products of quadratic algebras and their properties is  \cite[Section 3.2]{PoPo}. An interested reader may find results on Segre product of specific Artin-Schelter regular algebras in \cite{kristel}, and on twisted Segre product of noetherian Koszul Artin-Schelter regular algebras in \cite{twistedSegre}.

We first recall the notion of Segre product of graded algebras following
\cite[Ch 3 Sect 2, Def. 1]{PoPo}.
\begin{dfn}
\label{dfn:segre}
    Let \[
    A=  \textbf{k} \oplus A_1\oplus A_2   \oplus \cdots \; \mbox{ and } \; B=   \textbf{k} \oplus B_1\oplus B_2   \oplus \cdots\]
    be $\N$ graded algebras over a field $\textbf{k}$,  where $\textbf{k} = A_0 = B_0$.
The \emph{Segre product} of $A$ and $B$ is the $\N$-graded algebra
    \[A\circ B:=\bigoplus_{i \geq 0}(A\circ B)_i\; \text{with} \;  (A\circ B)_i =  A_i\otimes_{\textbf{k}} B_i. \]
\end{dfn}

The Segre product $A\circ B$ is a subalgebra of the tensor product algebra $A\otimes B$. Note that the embedding is not a graded algebra
morphism, as it doubles grading. If $A$ and $B$ are locally finite then the Hilbert function of $A\circ B$
    satisfies
    \begin{equation}
    \label{eq:hilbfunction}
    h_{A\circ B}(t)=\dim(A\circ B)_t=\dim(A_t\otimes B_t)=\dim(A_t)\cdot\dim(B_t)=h_A(t)\cdot h_B(t),
    \end{equation}
and for the Hilbert series one has
\[
H_A(t)=  \Sigma _{n \geq 0} (\dim A_n) t^n,  \; \; H_B(t)=  \Sigma _{n \geq 0} (\dim B_n) t^n, \;\;  H_{A\circ B}(t)=  \Sigma _{n \geq 0} (\dim A_n)(\dim B_n) t^n.
\]

The Segre product, $A \circ B$
 inherits various properties from the two algebras $A$ and $B$. In particular, if both algebras are one-generated, quadratic, and Koszul, it follows from \cite[Chap 3.2, Proposition 2.1]{PoPo}  that  the algebra
$A \circ B$ is also one-generated, quadratic, and Koszul.

The following remark gives more concrete information about the space of quadratic relations  of $A\circ B$, see for example,  \cite{kristel}.

\begin{rmk}
\label{rmk:relAoB1} \cite{kristel}
Suppose that $A$ and $B$ are quadratic algebras generated in degree one by $A_1$ and $B_1$, respectively, written as:
\[\begin{array}{ll}
A= T(A_1)/ (\Re_A) \quad &\text{with}\; \Re_A \subset A_1\otimes A_1,\\
B= T(B_1)/ (\Re_B)\quad &\text{with}\; \Re_B \subset B_1\otimes B_1,
\end{array}
\]
where $T(-)$ is the tensor algebra and $(\Re_A), (\Re_B)$ are the ideals of relations of $A$ and $B$.

Then  $A\circ B$ is also a quadratic algebra generated in degree one by $A_1\otimes B_1$ and presented as
\begin{equation}
\label{eq:segrelations1}
A \circ B = T(A_1\otimes B_1)/ (\sigma^{23}( \Re_A \otimes B_1\otimes B_1 +A_1\otimes A_1\otimes \Re_B)),
\end{equation}
where
\[
\sigma^{23}(a_1\otimes a_2\otimes b_1\otimes b_2)= a_1\otimes b_1\otimes a_2\otimes b_2.
\]
\end{rmk}
As usual,  we take a
combinatorial approach to study quadratic algebras. The properties of $A$ will be
read off a presentation $A= \textbf{k} \langle X\rangle /(\Re_A)$,
where by convention $X$ is a fixed finite set of generators of
degree one, $|X|=n,$
and $(\Re_A)$ is the two-sided
ideal of relations, generated by a {\em finite} set $\Re_A$ of
homogeneous polynomials of degree two.

\subsection{Segre products of Yang-Baxter algebras, generators and relations}
Suppose $(X,r_1)$ and $(Y, r_2)$ are finite solutions of orders $|X|= m$ and $|Y|= n.$

Let
 $A =\cA(\textbf{k},  X, r_1)$, and  $B =\cA(\textbf{k},  Y, r_2)$  be the corresponding YB-algebras.
As in Convention \ref{rmk:conventionpreliminary1} we fix enumerations
\[X = \{x_1, \cdots, x_m\}, \quad  \quad Y = \{y_1, \cdots, y_n\},\]
and consider the degree-lexicographic orders on the free monoids $\langle X\rangle$, and  $\langle Y\rangle$ extending these enumerations.
Then
\begin{equation}
\label{eq:relations_A}
\begin{array}{c}
A = \textbf{k}\langle X  \rangle /(\Re_1)\;
\text{where $\Re_1$ is a set of $\binom{m}{2}$ binomial relations}:\\
\Re_1= \{x_jx_i- x_{i^{\prime}}x_{j^{\prime}}\mid x_jx_i > x_{i^{\prime}}x_{j^{\prime}}\;\; \text{and} \;   r_1(x_jx_i)= x_{i^{\prime}}x_{j^{\prime}}\}.
\end{array}
\end{equation}

\begin{equation}
\label{eq:relations_B}
\begin{array}{c}
B = \textbf{k}\langle Y  \rangle /(\Re_2)\;
\text{where $\Re_2$ is a set of $\binom{n}{2}$ binomial relations}:\\
\Re_2=\{y_by_a-y_{a^{\prime}}y_{b^{\prime}}\mid
y_by_a> y_{a^{\prime}}y_{b^{\prime}} \;\; \text{and} \;\;r_2(y_by_a) = y_{a^{\prime}}y_{b^{\prime}}
 \}.
\end{array}
\end{equation}
One has
\begin{equation}
\label{eq:dimensions}
\dim A_2= \binom{m+1}{2}, \quad  \dim B_2= \binom{n+1}{2}, \quad \dim (A\circ B)_2 = \binom{m+1}{2} \binom{n+1}{2}.
\end{equation}
\begin{rmk}
\label{rmk:fixedpts2}
Note that if $(X,r)$ is a quadratic set, then $r(xy) = xy$ \emph{iff}  ${}^xy=x$ and $x^y =y$, $x, y \in X$. Moreover, if the monoid $S(X,r)$ is with cancellation,
then $r(xy) = xy$ is equivalent to  ${}^xy=x$.
\end{rmk}

Let $\Ncal(A)$ be the set of normal monomials modulo the ideal $(\Re_1)$  in $\textbf{k}\langle X\rangle$
and let $\Ncal(B)$ be the set of normal monomials modulo the ideal $(\Re_2)$  in $\textbf{k}\langle Y\rangle$.

\begin{rmk}
\label{rmk:N2}
\begin{enumerate}

\item
A monomial  $xy\in \Ncal(A)_2$,  $x, y \in X$ \emph{iff} either (a) ${}^xy>x,$ in this case $f={}^xy. x^y -xy \in \Re_1, HM(f) = {}^xy. x^y,$  or
(b) $r_1(xy)=xy$, which is equivalent to ${}^xy=x$, since the monoid $S(X,r_1)$ is cancellative, see Remark \ref{rmk:fixedpts2}.
\item
$zt\in \Ncal(B)_2, z, t \in Y$,  \emph{iff} either (a) ${}^zt>z,$ in this case $g={}^zt z^t -zt \in \Re_2, HM(g) = {}^zt z^t$,  or (b) $r_2(zt)=zt$,
which is equivalent to ${}^zt=z$.
\end{enumerate}
Thus
\[
\Ncal(A)_2= \{xy \in X^2\mid {}^xy \geq x\},      \quad \Ncal(B)_2= \{zt \in Y^2\mid {}^zt \geq z\}.
\]
\end{rmk}

\begin{dfn}
\label{dfn:cartesianproductof solutions}
    Let $(X,r_X)$ and $(Y, r_Y)$ be disjoint set-theoretic solutions of YBE. \emph{The Cartesian product of the solutions $(X,r_X)$ and $(Y, r_Y)$} is defined as
$(X\times Y, \rho)$,  $\rho=\rho_{X\times Y}$, where the map
\[\rho: (X\times Y)\times  (X\times Y) \longrightarrow (X\times Y)\times  (X\times Y) \]
is given by
\[
\rho= \sigma_{23}\circ (r_X\times r_Y)\circ \sigma_{23}, \; \text{and} \; \sigma_{23}\;  \text{is the flip of the second and the third component}.
\]
In other words
\begin{equation}
\label{eq:cartproduct}
\rho((x_j, y_b),(x_i, y_a)) := (({}^{x_j}{x_i},  {}^{y_b}{y_a}), (x_j^{x_i}, y_b^{y_a})),
\end{equation}
for all $i, j \in \{1, \cdots, m\}$ and all  $a, b \in \{1, \cdots, n\}$.
It is easy to see that the Cartesian product $(X\times Y, \rho_{X\times Y})$ is a solution of YBE of order $mn$.
\end{dfn}

\begin{rmk}
\label{rmk:cartesianproductof solproperties} The Cartesian product of solutions $(X\times Y, \rho_{X\times Y})$ satisfies the following conditions.
\begin{enumerate}
\item $(X\times Y, \rho_{X\times Y})$ is nondegenerate \emph{iff} $(X,r_X)$ and $(Y, r_Y)$  are nondegenerate.
\item  $(X\times Y, \rho_{X\times Y})$ is involutive \emph{iff} $(X,r_X)$ and $(Y, r_Y)$  are involutive.
\item  $(X\times Y, \rho_{X\times Y})$ is a square-free solution \emph{iff} $(X,r_X)$ and $(Y, r_Y)$  are  square-free solutions.
\end{enumerate}
\end{rmk}

To simplify notation when we work with elements of the Segre product $A\circ B$ we shall write $x\circ y$ instead of $x\otimes y$, whenever $x\in X, y\in Y$,
or $u\circ v$ instead of $u\otimes v$,  whenever $u \in A_k,$ $v \in B_k.$

\begin{pronotation}
\label{pronotation}
Let $(X, r_1)$ and $(Y, r_2)$ be solutions on the disjoint sets $X = \{x_1, \cdots , x_m\}$, and $Y = \{y_1, \cdots , y_n\}$.
 Let $A\circ B$ be the Segre product of the YB algebras $A= \cA(\textbf{k}, X, r_1)$ and $B=\cA(\textbf{k}. Y, r_2)$,
and let
\[X\circ Y =\{ x_i\circ y_a \mid  1 \leq i \leq m, \; 1\leq a \leq n\}.\]
There is a natural structure of a solution  $(X\circ Y, r_{X\circ Y})$ on the set $X\circ Y$,
where the map $r_{X\circ Y}$ is defined as
\begin{equation}
\label{eq:def_r}
r_{X \circ Y}((x_j \circ y_b), (x_i\circ y_a)) := (({}^{x_j}{x_i}\circ {}^{y_b}{y_a}),     ({x_j}^{x_i}\circ y_b^{y_a})),
\end{equation}
for all $1 \leq i, j\leq m$ and all $1 \leq a, b \leq n.$
The solution $(X\circ Y, r_{X\circ Y})$ is
isomorphic to the Cartesian product of solutions
$(X\times Y, \rho_{X\times Y}).$
In particular, the solution $(X\circ Y, r_{X\circ Y})$ has cardinality $mn$ and $\binom{mn}{2}$ nontrivial $r_{X\circ Y}$-orbits.
\end{pronotation}
\begin{proof}
The set $X\circ Y$ is a basis of $(A\circ B)_1 = A_1\otimes B_1$ and in particular it consists of $mn$ distinct elements.
The map $r: (X\circ Y )\times (X\circ Y )\longrightarrow (X\circ Y )\times (X\circ Y )$ defined via (\ref{eq:def_r}) is a well-defined bijection.
Consider the bijective map
\[F: X\circ Y\rightarrow X \times Y, \quad F(x\circ y) = (x, y) .\]
It follows from the definitions of the maps $\rho_{X\times Y}$ and $r_{X\circ Y}$
that
\[(F\times F)\cdot r_{X\circ Y}= \rho_{X\times Y}\cdot (F\times F).\]
Therefore $r_{X\circ Y}$ obeys the YBE, and $(X\circ Y, r_{X\circ Y})$ is a solution isomorphic to the Cartesian product of solutions
$(X\times Y, \rho_{X\times Y})$. In particular $(X\circ Y, r_{X\circ Y})$ is nondegenerate and involutive.
It is clear that $|X\circ Y|= mn$ and the solution $(X\circ Y, r_{X\circ Y})$ has $\binom{mn}{2}$ nontrivial $r_{X\circ Y}$-orbits.
\end{proof}
We shall often identify the solutions $(X\circ Y, r_{X\circ Y})$ and $(X\times Y, \rho_{X\times Y})$ and refer to $(X\circ Y, r_{X\circ Y})$ as "the Cartesian product of the solutions $(X, r_1)$ and $(Y, r_2)$".

\begin{pro}
\label{pro:relYB}
Let $(X, r_1)$ and $(Y, r_2)$ be solutions on the disjoint sets $X = \{x_1\cdots , x_m\}$, $Y = \{y_1\cdots , y_n\}$. In notation as above let $(X\circ Y, r=r_{X\circ Y})$ be the Cartesian product of the solutions  $(X, r_1)$ and $(Y. r_2)$. We order the set $X\circ Y$ lexicographically
$X \circ Y = \{x_1\circ y_1, \cdots, x_1\circ y_n, \cdots, x_m\circ y_n \}$.
%Then $(X\circ Y, r_{X\circ Y})$ is a solution (i.e. nondegenerate symmetric set).
The Yang-Baxter algebra
$\Alg = \Alg_{X\circ Y}= \cA(\textbf{k}, X\circ Y, r)$ is generated by the set $X \circ Y $  and has $\binom{mn}{2}$ quadratic defining relations
described in the two lists (\ref{eq:viprelations11})  and (\ref{eq:viprelations22}) below.
\begin{equation}
\label{eq:viprelations11}
\begin{array}{c}
f_{ji,ba}= (x_j\circ y_b) (x_i\circ y_a)- ({}^{x_j}{x_i}\circ {}^{y_b}{y_a})(x_j^{x_i}\circ y_b^{y_a}),\\
\text{for all}\; \;  1\leq i, j\leq m \; \; \text{such that}\;\; x_j > {}^{x_j}{x_i}, \;\; \text{and all}\; \;1 \leq a, b \leq n.
\end{array}
\end{equation}
Every relation $f_{ji,ba}$ has a leading monomial  $\LM (f_{ji,ba})= (x_j\circ y_b) (x_i\circ y_a)$.
\begin{equation}
\label{eq:viprelations22}
\begin{array}{c}
f_{ij,ba}= (x_i\circ y_b) (x_j\circ y_a)- (x_i\circ {}^{y_b}{y_a})(x_j\circ y_b^{y_a}),\\
\text{for all}\;\;  1\leq i, j\leq m \;\; \text{with}\;\; r_1( x_ix_j) = x_ix_j ,\; \; \text{and all}\; \; 1 \leq a, b \leq n, \;\; \text{such that}\;\;
y_b >{}^{y_b}{y_a} .
\end{array}
\end{equation}
Every relation $f_{ij,ba}$ has a leading monomial  $\LM (f_{ij,ba})= (x_i\circ y_b) (x_j\circ y_a)$.

The solution $(X\circ Y, r)$ has exactly $mn$ fixed points, namely:
\[
\cF = \{(x_p\circ y_a)(x_q\circ y_b)\mid r_1(x_px_q) = x_px_q, \; p, q \in \{1. \cdots, m\}, \;\text{and}\; r_2(y_ay_b) = y_ay_b, \;a, b \in \{1, \cdots, n\}\}
\]
In this case $x_px_q\in \Ncal(A)_2$  and $y_ay_b \in \Ncal(B)$.
\end{pro}
\begin{proof}
The solution $(X\circ Y, r)$ is nondegenerate, it has order $|X\circ Y|= mn$, and therefore, by Corollary \ref{cor:orbits_and_fixedpoints}
the number of its fixed points is $mn$. It is clear that $r((x_p\circ y_a)(x_q\circ y_b)) = (x_p\circ y_a)(x_q\circ y_b)$ if and only if $r_1(x_px_q) = x_px_q$ and $r_2(y_ay_b) = y_ay_b$.
The defining relations of the Yang-Baxter algebra $\Alg$ correspond to the nontrivial $r$-orbits.
There are exactly $\binom{m}{2}n^2$ distinct relations given in (\ref{eq:viprelations11}), each of them corresponds to a pair $(x_j\circ y_b, x_i\circ y_a)$, where $x_jx_i> r_1(x_jx_i)$, and $y_by_a$ is an arbitrary word in $Y^2$. There are exactly $m\binom{n}{2}$ distinct relations in
(\ref{eq:viprelations22}), each of them is determined by a fixed point $x_ix_j$ in $X^2$ and some nontrivial $r_2$-orbit in $Y^2$.
Note that
\[\binom{m}{2}n^2 + m\binom{n}{2} = \binom{mn}{2},\]
as desired.
\end{proof}

The next corollary is a straightforward consequence from \cite[Chap 3, Proposition 2.1]{PoPo} and Facts \ref{fact2}.
\begin{cor}
\label{cor:SegreProductProperties}
Let $(X, r_1)$ and $(Y, r_2)$ be finite solutions and let $A= \cA(\textbf{k}, X, r_1)$ and $B=\cA(\textbf{k}, Y, r_2)$ be their Yang-Baxter algebras. Then the Segre product, $A \circ B$
is a one-generated quadratic and Koszul algebra.
\end{cor}
Moreover,  $A \circ B$ is also a left and a right Noetherian algebra with polynomial growth, see Corollary \ref{cor:SegreProdNoetherian}.

 \begin{thm}
\label{thm:rel_segre}
 Let $(X, r_1)$ and $(Y, r_2)$ be finite solutions, where $X = \{x_1\cdots , x_m\}$ and $Y = \{y_1\cdots , y_n\}$  are disjoint sets.
 Let $A\circ B$ be the Segre product of the YB algebras $A= \cA(\textbf{k}, X, r_1)$ and $B=\cA(\textbf{k}, Y, r_2)$, and let $(X\circ Y, r_{X\circ Y})$ be the solution from Proposition \ref{pronotation}.

The algebra $A\circ B$ has a set of $mn$ one-generators $W=X\circ Y$
ordered lexicographically:
\begin{equation}
\label{eq:W}
W = \{w_{11}= x_1\circ y_1< w_{12}= x_1\circ y_2 <\cdots <w_{1n}= x_1\circ y_n< w_{21}= x_2\circ y_1< \cdots < w_{mn}= x_m\circ y_n\},
\end{equation}
 and a set of
$\binom{mn}{2}+\binom{m}{2}\binom{n}{2}$ linearly independent quadratic relations $\Re$. The set $\Re$ splits as a disjoint union
$\Re= \Re_a \cup \Re_b$, where the sets $\Re_a$ and $\Re_b$ are described below.

\begin{enumerate}
\item The set $\Re_a$ is a disjoint union  $\Re_a= \Re_{a1} \cup
\Re_{a2}$ of two sets described as follows.
\[
\begin{array}{ll}
\Re_{a1}= &\{f_{ji,ba}=  (x_j\circ y_b) (x_i\circ y_a)- (x_{i^{\prime}}\circ y_{a^{\prime}})(x_{j^{\prime}}\circ y_{b^{\prime}}), \; 1 \leq i,j\leq m, 1 \leq a,b\leq n,\\
&\text{where}\; r_1(x_jx_i)=x_{i^{\prime}}x_{j^{\prime}}, \;\text{with} \;   j > i^{\prime}, \;
%\text{or equivalently},\; r_1(x_jx_i)< x_jx_i,
\text{and} \;  r_2(y_by_a) = y_{a^{\prime}}y_{b^{\prime}}\}.
\end{array}
\]
Every relation $f_{ji,ba}$ has leading monomial  $\LM (f_{ji,ba})= (x_j\circ y_b) (x_i\circ y_a)$.
The cardinality of $\Re_{a1}$ is $|\Re_{a1}|=\binom{m}{2}n^2$.

\[
\begin{array}{ll}
\Re_{a2}= &\{f_{ij,ba}= (x_i\circ y_b) (x_j\circ y_a)- (x_{i}\circ y_{a^{\prime}})(x_{j}\circ y_{b^{\prime}}), \; 1 \leq i,j\leq m,\; 1 \leq a,b\leq n,\\
 &\text{where}\;  x_ix_j = r_1( x_ix_j) \;\text{is a fixed point}\; \text{and }\; r_2(y_by_a) = y_{a^{\prime}}y_{b^{\prime}},\; \text{with}\; b > a^{\prime}\}.
\end{array}
\]
Every relation $f_{ij,ba}$ has leading monomial  $\LM (f_{ij,ba})= (x_i\circ y_b) (x_j\circ y_a)$.
The cardinality of $\Re_{a2}$ is $|\Re_{a2}|=m\binom{n}{2}$.

%In fact the relations $\Re_a$  are exactly the defining relations of the YB-algebra $\Alg_{X\circ Y}= \cA(\textbf{k}, X\circ Y, %r_{X\circ Y})$,
%there is a one-to one correspondence between the set of relations $\Re_a$ and the set of nontrivial $r_{X\circ Y}$-orbits in
%$(X\circ Y) \times (X\circ Y)$.
\item
The set $\Re_b$ consists of $\binom{m}{2}\binom{n}{2}$ relations given explicitly in (\ref{eq:viprelations3})
\begin{equation}
\label{eq:viprelations3}
\begin{array}{ll}
\Re_{b}= &\{g_{ij,ba}= (x_i\circ y_b) (x_j\circ y_a)- (x_{i}\circ y_{a^{\prime}})(x_{j}\circ y_{b^{\prime}}), \; 1 \leq i,j\leq m, 1 \leq a,b\leq n,\\
&\text{where}\; r_1( x_ix_j) > x_ix_j , \; r_2(y_by_a) = y_{a^{\prime}}y_{b^{\prime}} \;
 \text{and}\;\;b >a^{\prime}.
\end{array}
\end{equation}
 Every relation $g_{ij, ba}$ has leading monomial  $\LM (g_{ij, ba})= (x_i\circ y_b) (x_j\circ y_a)$.
\end{enumerate}

 \end{thm}

\begin{proof}

Note that the relations in $\Re_a$ are the same as the defining relations of the Yang-Baxter algebra
$\Alg_{X\circ Y}= \cA(\textbf{k}, X\circ Y, r_{X\circ Y})$  from Proposition
\ref{pro:relYB}. There is an obvious 1-1 correspondence between the set of nontrivial $ r_{X\circ Y}$- orbits in $X\circ Y \times X\circ Y$
and the set of relations $\Re_a$.  Being a nondegenerate symmetric set of order $mn$, $(X\circ Y, r_{X\circ Y})$ has exactly $\binom{mn}{2}$ nontrivial
$ r_{X\circ Y}$-orbits and therefore the cardinality of $\Re_a$ must satisfy
\begin{equation}
\label{eq:cardRea}
|\Re_a| =\binom{mn}{2}.
\end{equation}
It is clear that $\Re_{a1}$ and $\Re_{a2}$ are disjoint subsets of $\Re_{a}$. To be sure that the sets $\Re_{a1}$ and $\Re_{a2}$ exhaust $\Re_{a}$ we count their cardinalities.

Each of the relations $f_{ji, ba} \in \Re_{a1}$
 corresponds to a pair $(x_j\circ y_b, x_i\circ y_a)$, where $x_jx_i> r_1(x_jx_i)$, and $y_by_a$ is an arbitrary word in $Y^2$. There are exactly $\binom{m}{2}n^2$ distinct elements of this type.

 Each of the relations $f_{ij, ba} \in \Re_{a2}$
is determined by a fixed point $x_ix_j$ in $X^2$ and some nontrivial $r_2$-orbit in $Y^2$, $\{y_by_a, y_{a^{\prime}}y_{b^{\prime}}=r_2(y_by_a)\}$ with $b > a^{\prime}$. There are $m\binom{n}{2}$ distinct elements of this type.
One has
\[|\Re_{a1}|+|\Re_{a2}|= \binom{m}{2}n^2 + m\binom{n}{2} = \binom{mn}{2}=|\Re_{a}|,\]
as desired.

We shall prove now that the sets $\Re_a$ and $\Re_b$ described above are contained in the ideal of relations
$(\Re(A\circ B))$ of $A\circ B$.
Under the hypothesis of the theorem we prove the following lemma.

 \begin{lem}
\label{lem:segreproductrel}
\begin{enumerate}
\item
\label{lem:segreproductrel1}
Suppose $f=x_jx_i-x_{i^{\prime}}x_{j^{\prime}}\in \Re_1$, with $HM(f) = x_jx_i$. Let $y_b, y_a \in Y$,  and let
$r_2(y_by_a)=y_{a^{\prime}}y_{b^{\prime}}$ (it is possible that $y_by_a$ is a fixed point, or $y_by_a < y_{a^{\prime}}y_{b^{\prime}}$).  Then
\[
f_{ji,ba}= (x_j\circ y_b)(x_i\circ y_a) - (x_{i^{\prime}}\circ y_{a^{\prime}})  (x_{j^{\prime}}\circ y_{b^{\prime}}) \in (\Re(A\circ B)).
\]
Moreover, the relation $f_{ji,ba}$ has leading monomial  $\LM (f_{ji,ba})= (x_j\circ y_b) (x_i\circ y_a)$.

\item
\label{lem:segreproductrel2}
Suppose $r_1( x_ix_j) = x_ix_j $ (that is $x_ix_j$ is a fixed point), for some $1\leq i, j \leq m$, and let $r(y_by_a) =y_a^{\prime}y_b^{\prime}$,
with $y_b >y_{a^{\prime}}$, (so $y_by_a -y_a^{\prime}y_b^{\prime} \in \Re_2$). Then
\[
f_{ij,ba}= (x_i\circ y_b)(x_j\circ y_a) - (x_{i}\circ y_{a^{\prime}})  (x_{j}\circ y_{b^{\prime}}) \in (\Re(A\circ B)),
\]
and  $\LM (f_{ij,ba})= (x_i\circ y_b) (x_j\circ y_a)$.
\item
\label{lem:segreproductrel3}
If $r_1(x_ix_j) > x_ix_j$, and $y_by_a - y_{a^{\prime}} y_{b^{\prime}} \in \Re_2$, then
\[
g_{ij,ba}= (x_i\circ y_b)(x_j\circ y_a) - (x_i\circ y_{a^{\prime}})  (x_j\circ y_{b^{\prime}}) \in (\Re(A\circ B)),
\]
and $\LM (g_{ij,ba})= (x_i\circ y_b) (x_j\circ y_a)$.
\end{enumerate}
\end{lem}
\begin{proof}
(\ref{lem:segreproductrel1}) By hypothesis
$x_jx_i-x_{i^{\prime}}x_{j^{\prime}}\in \Re_1$ and $y_by_a-y_{a^{\prime}}y_{b^{\prime}}$ is in the ideal $(\Re_2).$
Then, by Remark \ref{rmk:relAoB1}
\[
\begin{array}{ll}
\varphi &=\sigma_{23}((x_jx_i-  x_{i^{\prime}}x_{j^{\prime}})\circ(y_by_a))\\
            &=(x_j\circ y_b)(x_i\circ y_a)-
(x_{i^{\prime}}\circ y_b)(x_{j^{\prime}}\circ y_a)
\in (\Re(A\circ B))\\
\psi &= \sigma_{23}((x_{i^{\prime}}x_{j^{\prime}})\circ(y_by_a - y_{a^{\prime}} y_{b^{\prime}}))\\
       &=(x_{i^{\prime}}\circ y_b)  (x_{j^{\prime}}\circ y_a) -
            (x_{i^{\prime}}\circ   y_{a^{\prime}})  (x_{j^{\prime}}\circ y_{b^{\prime}} ) \in (\Re(A\circ B)).
\end{array}
\]
The elements $\varphi$ and $\psi$ are in the ideal of relations  $(\Re(A\circ B))$, so the sum $\varphi + \psi$ is also in
$(\Re(A\circ B))$. One has
\[
\varphi + \psi =
(x_j\circ y_b)(x_i\circ y_a)- (x_{i^{\prime}}\circ   y_{a^{\prime}})  (x_{j^{\prime}}\circ y_{b^{\prime}} ) = f_{ji,ba} \in (\Re(A\circ B)).
\]
By definition, $f=x_jx_i-x_{i^{\prime}}x_{j^{\prime}}\in \Re_1$ \emph{iff}
$r_1(x_jx_i) = x_{i^{\prime}}x_{j^{\prime}}$ and $x_jx_i >x_{i^{\prime}}x_{j^{\prime}}$. The cancellation low implies that $x_j >x_{i^{\prime}}$.
Thus $(x_{i^{\prime}}\circ   y_{a^{\prime}})  (x_{j^{\prime}}\circ y_{b^{\prime}} )  <( x_j\circ y_b)(x_i\circ y_a)$, and
$\LM (f_{ji,ba})= (x_j\circ y_b) (x_i\circ y_a)$.

(\ref{lem:segreproductrel2})   By hypothesis $y_by_a - y_{a^{\prime}} y_{b^{\prime}} \in \Re_2$, then by
Remark \ref{rmk:relAoB1} again,
\[
\begin{array}{ll}
f_{ij,ba}&=\sigma_{23}((x_ix_j)\circ(y_by_a-y_{a^{\prime}} y_{b^{\prime}}))\\
            &=(x_i\circ y_b)(x_j\circ y_a)-  (x_i\circ y_{a^{\prime}})(x_j\circ  y_{b^{\prime}})
\in (\Re(A\circ B)).
\end{array}
\]
It is clear that $HM (f_{ij,ba})= (x_i\circ y_b)(x_j\circ y_a)$ which proves (2).

(\ref{lem:segreproductrel3}).
Suppose  $y_by_a - y_{a^{\prime}} y_{b^{\prime}} \in \Re_2$, and $r_1(x_ix_j) > x_ix_j$. Then
$r_1(x_ix_j) =x_{j^{\prime}}x_{i^{\prime}}$ and $r_1 (x_{j^{\prime}}x_{i^{\prime}})=x_ix_j$ for some $1 \leq j^{\prime}, i^{\prime} \leq m$, so
$x_{j^{\prime}}x_{i^{\prime}}-x_ix_j\in \Re_1.$
By
Remark \ref{rmk:relAoB1}

\[
\begin{array}{ll}
\varphi_1 &=\sigma_{23}(x_{j^{\prime}}x_{i^{\prime}}-x_ix_j)\circ(y_by_a))\\
            &=(x_{j^{\prime}}\circ y_b)(x_{i^{\prime}}\circ y_a)-
(x_{i}\circ y_b)(x_{j}\circ y_a)
\in (\Re(A\circ B)).
\end{array}
\]

By part (1)
\[
f_{j^{\prime}i^{\prime},ba}= (x_{j^{\prime}}\circ y_b)(x_{i^{\prime}}\circ y_a) - (x_{i}\circ y_{a^{\prime}})  (x_{j}\circ y_{b^{\prime}}) \in (\Re(A\circ B)).
\]

It follows that $f_{j^{\prime}i^{\prime},ba} - \varphi_1\in (\Re(A\circ B)).$ The explicit computation gives
\[
\begin{array}{ll}
f_{j^{\prime}i^{\prime},ba} - \varphi_1 &= (x_{j^{\prime}}\circ y_b)(x_{i^{\prime}}\circ y_a) -
(x_{i}\circ y_{a^{\prime}})  (x_{j}\circ y_{b^{\prime}})- (x_{j^{\prime}}\circ y_b)(x_{^{\prime}}\circ y_a)+
(x_{i}\circ y_b)(x_{j}\circ y_a)\\
&= (x_{i}\circ y_b)(x_{j}\circ y_a)- (x_{i}\circ y_{a^{\prime}})  (x_{j}\circ y_{b^{\prime}})=g_{ij,ba}.
\end{array}
\]
We have shown that $g_{ij,ba}\in (\Re(A\circ B)).$ It is clear that $\LM (g_{ij,ba})= (x_i\circ y_b) (x_j\circ y_a)$.
\end{proof}
Note that the sets $\Re_a$ and $\Re_b$ consist of quadratic polynomials in the set $X\circ Y$ of one-generators of $A\circ B$.
It follows from Lemma \ref{lem:segreproductrel} that every element of $\Re =\Re_a \cup \Re_b$ is a relation of $A\circ B$.

We have to show that the elements of $\Re$ form a basis of the ideal of relations $(\Re(A\circ B))$.
It will be convenient to use the description of $\Re_a$ and $\Re_b$ as sets of quadratic polynomials in the variables $W$, see (\ref{eq:W}), so we simply replace $x_i\circ y_a$ by $w_{ia}$ in each of the relations in $\Re$.

\begin{rmk}
\label{rmk:relinW}
Theorem \ref{thm:rel_segre}  states that the Segre product $A\circ B$ has a finite presentation
\[
A\circ B \simeq \textbf{k} \langle w_{11}, \cdots, w_{mn}\rangle/(\Re) ,
\]
where $\Re$ is a set of $\binom{mn}{2}+\binom{m}{2}\binom{n}{2}$  quadratic polynomials in the free associative algebra
$\textbf{k}\langle w_{11}, \cdots, w_{mn}\rangle$. More precisely, $\Re$ is a disjoint union
 $\Re= \Re_a \cup \Re_b$ of the sets $\Re_a$ and $\Re_b$ described below.
\begin{enumerate}
\item The set $\Re_a$ consists of $\binom{mn}{2}$ relations given explicitly in (\ref{eq:viprelations1a})  and  (\ref{eq:viprelations2a}) :
\begin{equation}
\label{eq:viprelations1a}
\begin{array}{c}
f_{ji,ba}= w_{jb}w_{ia}-w_{i^{\prime}a^{\prime}}w_{j^{\prime}b^{\prime}}, 1\leq i, j \leq m, 1 \leq a, b \leq n,
\\
\text{where}\;
r_1(x_jx_i)  = x_{i^{\prime}}x_{j^{\prime}}, \;j > i^{\prime}\; \text{and} \; r_2(y_by_a)=y_{a^{\prime}}y_{b^{\prime}}.
\end{array}
\end{equation}
Every relation $f_{ji,ba}$ has leading monomial  $\LM (f_{ji,ba})= w_{jb}w_{ia}$.

\begin{equation}
\label{eq:viprelations2a}
\begin{array}{c}
f_{ij,ba}= w_{ib}w_{ja}- w_{ia^{\prime}} w_{jb^{\prime}}, 1\leq i,j\leq m, 1 \leq a,b\leq n,\\
 \text{where}\; r_1( x_ix_j) = x_ix_j,  \text{and}\;
r_2(y_by_a) = y_{a^{\prime}}y_{b^{\prime}}\; \text{with}\;\; b > a^{\prime}.
\end{array}
\end{equation}
Every relation $f_{ij,ba}$ has leading monomial  $\LM (f_{ij,ba})= w_{ib}w_{ja}$.

\item
The set $\Re_b$ consists of $\binom{m}{2}\binom{n}{2}$ relations given explicitly in (\ref{eq:viprelations3b})
\begin{equation}
\label{eq:viprelations3b}
\begin{array}{c}
g_{ij, ba}= w_{ib}w_{ja}- w_{ia^{\prime}} w_{jb^{\prime}},   1\leq i,j\leq m, 1 \leq a,b\leq n, \\
\text{where}\; r_1( x_ix_j) > x_ix_j ,\; \;\text{and}\; \; r_2(y_by_a) = y_{a^{\prime}}y_{b^{\prime}}\;\;
 \text{with}\;\;b >a^{\prime}.
\end{array}
\end{equation}
 Every relation $g_{ij,ba}$ has leading monomial  $\LM (g_{ij,ba})= w_{ib}w_{ja}$.
\end{enumerate}
\end{rmk}

%%%

%Some notes on the cardinalities of $\Re_a$ and $\Re_b.$ Consider the solution $(X\circ Y, r_{X\circ Y})$.

Next we count the relations in $\Re_b$. The number of $x_ix_j, 1 \leq i, j \leq m$ such that $r(x_ix_j) > x_ix_j$ is exactly the number of nontrivial $r_1$-orbits in $X \times X$  which is $\binom{m}{2}$. The number of pairs $y_b,y_a$ with $y_by_a > r_2(y_by_a)$ equals the number of nontrivial $r_2$ -orbits,
which is $\binom{n}{2}$ ,
hence
\begin{equation}
\label{eq:cardReb}
|\Re_b| =\binom{m}{2}\binom{n}{2}.
\end{equation}
The two sets $\Re_a$ and $\Re_b$ are disjoint. Indeed, the leading monomials of all elements in $\Re$ are pairwise distinct, and therefore the relations are pairwise distinct.
So $\Re=\Re_a\cup \Re_b$ is a disjoint union of sets, and by  (\ref{eq:cardRea})  and  (\ref{eq:cardReb})  one has:
\begin{equation}
\label{eq:cardRe}
|\Re|=|\Re_a|+|\Re_b| =\binom{mn}{2} +\binom{m}{2}\binom{n}{2}.
\end{equation}
It remains to show that $\Re$ is a linearly  independent set.

\begin{lem}
    \label{lem:independence}
Under the hypothesis of Theorem \ref{thm:rel_segre},   the set of
polynomials $\Re \subset \textbf{k}\langle W \rangle$ is linearly independent.
\end{lem}
\begin{proof}
This proof is routine. Note that the set of all words in $\langle W\rangle$ forms a basis of the free associative algebra
$\textbf{k}\langle W\rangle$ (considered as a vector space), in particular every finite set
of distinct words in $\langle W\rangle$ is linearly independent. Consider the presentation of $\Re$ given in Remark \ref{rmk:relinW}.
All words occurring in
$\Re$ are monomials of length 2 in $W^{2}$, but some of them occur in more than one
relation, e.g. the leading monomial $w_{ib}w_{ja}$, of $g_{ijba}$  occurs as a second monomial in some $f$ given in (\ref{eq:viprelations1a}).
Indeed, there is unique pair $j_1, i_1$ such that $r(x_{j_1}x_{i_1}) =x_ix_j$, $j_1> i$. It is clear that $r_2(y_{a^{\prime}}y_{b^{\prime}}) =y_by_a$ (since $r_2$ is involutive). Then, by definition
\[f_{j_1i_1, a^{\prime}b^{\prime}} = (x_{j_1}\circ y_{a^{\prime}})( x_{i_1}\circ y_{b^{\prime}}) - (x_i \circ y_b)(x_j\circ y_a)=
w_{j_1a^{\prime}}w_{i_1b^{\prime}}- w_{ib}w_{ja}. \]

We shall prove the lemma in three steps.

\begin{enumerate}
\item
  The set of polynomials $\Re_{a}\subset \textbf{k}\langle w_{11}, \cdots,   w_{mn}\rangle$ is linearly
  independent.
We have noticed that the polynomials in $\Re_a$ are in 1-to-1 correspondence with the
nontrivial $r_{X\circ Y}$-oprbits in $(X\circ Y)\times (X\circ Y)$. The orbits are disjoint and therefore
the relations $\Re_a$ involve exactly $m^2n^2-mn$ distinct monomials in $W^2$.
A linear relation
\[
\sum_{f \in \Re_a}\alpha_{f} f = 0, \text{where
all}\; \alpha_{f} \in \textbf{k},
 \]
involves only pairwise distinct monomials in $W^{2}$ and therefore it must be
trivial: $\alpha_f= 0$, for all $f \in \Re_a$.
It follows that $\Re_a$ is linearly independent.

\item
The set  $\Re_{b} \subset \textbf{k}\langle W \rangle$ is linearly independent.

Assume the contrary.  Then there exists a nontrivial linear relation for the elements of $\Re_b:$
\begin{equation}
\label{eq:Rb1a}
\sum_{g \in \Re_b}\beta_{g} g = 0,\; \text{with all}\; \beta_{g}
\in \textbf{k}.
\end{equation}
Let $g_{ij,ba}$ be the polynomial with $\beta_{g_{ij,ba}}\neq 0$ whose leading monomial is the highest among all leading monomials of polynomials
$g\in \Re_b$, with $\beta_g \neq 0$. So we have

\begin{equation}
\label{eq:Rb2a}
\\LM(g_{ij,ba}) = w_{ib}w_{ja} > \LM(g), \; \text{for all}\; g \in \Re_b,\; \text{where}\; \beta_{g} \neq 0.
\end{equation}
We use (\ref{eq:Rb1a})  to find the following equality in $\textbf{k}\langle W \rangle$:
\[
w_{ib}w_{ja} =   w_{ia^{\prime}}w_{jb^{\prime}} -\sum_{g \in \Re_b, \;
\LM(g)< w_{ib}w_{ja}} \; \frac{\beta_g}{\beta_{g_{ij,ba}}}\;g.
\]
It follows from  (\ref{eq:Rb2a})  that the right-hand side of this equality  is a linear combination of monomials strictly smaller  than $w_{ib}w_{ja} $ (in the lexicographic order on $\langle W \rangle$ ), which is impossible.
It follows that the set $\Re_{b} \subset \textbf{k}\langle W \rangle$ is linearly independent.

\item The set  $\Re \subset \textbf{k}\langle W \rangle$ is linearly independent.

Assume that the polynomials in $\Re$ satisfy a linear relation
\begin{equation}
  \label{eq:linindept1a}
\sum_{f \in \Re_a}\alpha_{f} f +\sum_{g \in \Re_b
}\beta_{g} g = 0, \;\text{where all}\;\; \alpha_{f},
\beta_{g} \in \textbf{k}.
\end{equation}
Every $f \in \Re_a$ can be written $f= u_f - u^{\prime}_f$, where $u_f, u^{\prime}_f\in W^2, u_f > u^{\prime}_f.$
Similarly, every $g \in \Re_b$ is $g= u_g - u^{\prime}_g$, where $u_g, u^{\prime}_g\in W^2, u_g > u^{\prime}_g$.
This gives the following equality in the free associative algebra
$\textbf{k}\langle W \rangle$ :

\begin{equation}
  \label{eq:linindept2a}
S_1 =\sum_{f\in \Re_a}\alpha_{f} u_f =
\sum_{f\in \Re_a}\alpha_{f}   u^{\prime}_f
-\sum_{g\in \Re_b}\beta_{g} g=S_2.
\end{equation}
The element $S_1 =\sum_{f   \in \Re_a}\alpha_{f} u_f $ on the
left-hand side of (\ref{eq:linindept2a}) is in the space
$V_1=\Span B_1$, where $B_1=\LM(\Re_a)=\{ u_f \mid f \in \Re_a\}$ is linearly
independent since it consists of distinct monomials. The element $S_2$
on the right-hand side of the equality is in the space $V_2= \Span B,$
where
\[
B=\{ u_f^{\prime}\mid f \in \Re_a\}
\cup
\{\text{all monomials} \; u_g, u^{\prime}_g \mid g\in \Re_b \}.
\]
Take a subset $B_2 \subset B$ which forms a basis of $V_2$. Note that
$B_1\cap B = \emptyset,$ hence $B_1\cap B_2=\emptyset.$
Moreover each of the sets $B_1$, and $B_2$ consists of pairwise distinct
monomials in $W^2$ and it is easy to show that
$V_1\cap V_2 = \{0\}.$  Thus the equality $S_1 = S_2 \in V_1\cap V_2 = \{0\}$ implies a linear
relation
\[S_1 =\sum_{f\in \Re_a}\alpha_{f} u_f
= 0, \]
for the set $B_1$ of leading monomials of $\Re_a$. But $B_1$ consists of  pairwise distinct monomials, and
therefore it is linearly independent. It follows that   all coefficients $\alpha_f,  f \in \Re_a$ equal $0$.
 This together with (\ref{eq:linindept1a}) implies the
linear relation
\[\sum_{g \in \Re_b}\beta_{g} g = 0,\]
 and since by (2) $\Re_b$ is linearly independent we get again $\beta_{g}= 0,
 \forall  \; g \in \Re_b $.
It follows that
the linear relation (\ref{eq:linindept1a}) must be trivial, and therefore
$\Re$ is a linearly independent set of quadratic polynomials in $\textbf{k}\langle W \rangle$.
\end{enumerate}
%This proves the lemma.
\end{proof}
We claim that $\Re$ is a set of defining relations for $A\circ B$. We know that $A\circ B$ is a quadratic algebra, that is its ideal of relations is generated by homogeneous polynomials of degree $2$, see Corollary \ref{cor:SegreProductProperties}.

Consider the graded ideal $J =(\Re)$ of $\textbf{k}\langle W\rangle$. To show that  that $J=(\Re(A\circ B))$ it will be enowgh to verify that there is an isomorphism of vector spaces:
\[
(\Re)_2 \oplus (A\circ B)_2 = (\textbf{k}\langle W\rangle)_2 ,
\]
or equivalently
\begin{equation}
\label{eq:dimensions11}
\dim \Span_{\textbf{k}}\Re + \dim_{\textbf{k}}  (A\circ B)_2 = \dim_{\textbf{k}} (\textbf{k}\langle W\rangle)_2.
\end{equation}
We have shown that $\Re$ is linearly independent, so $\dim \Span_{\textbf{k}}\Re =|\Re|= \binom{mn}{2}+\binom{m}{2}\binom{n}{2}$.
We also know that $\dim_{\textbf{k}}  (A\circ B)_2 = \dim_{\textbf{k}} A_2  \dim_{\textbf{k}} B_2 = \binom{m+1}{2}\binom{n+1}{2}$. Then

\[
\dim Span_{\textbf{k}}\Re + \dim_{\textbf{k}}  (A\circ B)_2 = \binom{mn}{2}+\binom{m}{2}\binom{n}{2} + \binom{m+1}{2}\binom{n+1}{2} = m^2n^2= \dim_{\textbf{k}} (\textbf{k}\langle W\rangle)_2,
\]
as desired.
It follows that $\Re$ is a set of defining relations for the Segre product $A\circ B$.
\end{proof}

\section{Segre maps of Yang-Baxter algebras}
\label{Sec:SegreMaps}
In this section we introduce and investigate
non-commutative analogues of the Segre maps in the class of Yang-Baxter algebras of finite solutions.
Our main result is Theorem \ref{thm:segremap}. As a consequence Corollary \ref{cor:SegreProdNoetherian} shows that the Segre product $A\circ B$ of two Yang-Baxter algebras $A$ and $B$ is always left and right Noetherian.
The results agree with their classical analogues in the commutative case, \cite{Harris}.

We keep the conventions and notation from the previous sections.
As usual $(X, r_1)$ and $(Y, r_2)$ are disjoint solutions of finite orders $m$ and $n$, respectively,
 $A =\cA(\textbf{k},  X, r_1)$, and  $B =\cA(\textbf{k},  Y, r_2)$  are the corresponding YB-algebras.
We fix enumerations
\[
X = \{x_1, \cdots, x_m\},  \quad Y = \{y_1, \cdots, y_n\},
\]
as in Convention \ref{rmk:conventionpreliminary1} and consider the degree-lexicographic orders on the free monoids $\langle X\rangle$, and $\langle Y\rangle$ extending these enumerations.
$A\circ B$ is the Segre product of $A$ and $B$, its set of one-generators is
\[W=X\circ Y =\{w_{11}=x_1\circ y_1 < w_{12}= x_1\circ y_2 <\cdots <w_{1n}= x_1\circ y_n< w_{21}= x_2\circ y_1<\cdots w_{mn}=x_m\circ y_n \},\]
ordered lexicographically, and $(X\circ Y, r_{X\circ Y})$ is the solution isomorphic to the Cartesian product $(X\times Y, \rho_{X\times Y})$, see Proposition-Notation \ref{pronotation}.

\begin{defnotation}
\label{def:Z}
Let $Z=\{z_{11}, z_{12}, \cdots, z_{mn}\}$ be a set of order $mn$, disjoint with $X$ and $Y$. Define a map
\[r=r_Z: Z\times Z\longrightarrow Z\times Z
\]
induced canonically from the solution
$(X\circ Y, r_{X\circ Y})$:
\[
r(z_{jb},z_{ia})=(z_{i^{\prime}a^{\prime}}, z_{j^{\prime}b^{\prime}})\;\text{\emph{iff}}\; r_{X\circ Y}(x_j\circ y_b, x_i\circ y_a)=
(x_{i^{\prime}}\circ y_{a^{\prime}}, x_{j^{\prime}}\circ y_{b^{\prime}}).
\]
It is clear, that $(Z, r_Z)$ is a solution isomorphic to $(X\circ Y, r_{X\circ Y})$ (and isomorphic to the Cartesian product
$(X\times Y, r_{X\times Y}$)).
\end{defnotation}
We consider the degree-lexicographic order on the free monoid $\langle Z\rangle$ induced by the enumeration of $Z$
\[Z= \{z_{11}< z_{12}< \cdots <  z_{mn}\}.\]
\begin{rmk}
\label{remark}
Let  $\mathbb{A}_Z= \cA(\textbf{k},  Z, r_Z)$ be the  YB-algebra of  $(Z, r_Z)$.
Then $\mathbb{A}_Z =\textbf{k}\langle Z\rangle/(\Re(\mathbb{A}_Z))$, where the ideal of relations of $\mathbb{A}_Z$
is generated by the the set $\Re(\mathbb{A}_Z)$ consisting of
$\binom{mn}{2}$ quadratic binomial relations given explicitly in (\ref{eq:viprelationsZ1})  and  (\ref{eq:viprelationsZ2}) :
\begin{equation}
\label{eq:viprelationsZ1}
\begin{array}{c}
\varphi_{ji,ba}= z_{jb}z_{ia}- z_{i^{\prime}a^{\prime}}z_{j^{\prime}b^{\prime}},
1 \leq i,j\leq m, 1 \leq a,b\leq n,  \\
\text{where}\;r_Z(z_{jb}z_{ia}) = z_{i^{\prime}a^{\prime}}z_{j^{\prime}b^{\prime}}, \;\text{and}\;
j >i^{\prime},
 \text{or equivalently}, \;r_Z(z_{jb}z_{ia})< z_{jb}z_{ia}.
\end{array}
\end{equation}
Every relation $\varphi_{ji,ba}$ has leading monomial  $\LM (\varphi_{ji,ba})=z_{jb}z_{ia}$.

\begin{equation}
\label{eq:viprelationsZ2}
\begin{array}{c}
\varphi_{ij,ba}= z_{ib}z_{ja}- z_{i a^{\prime}}z_{jb^{\prime}}, \;
1 \leq i, j\leq m, 1 \leq a, b\leq n, \\
 \text{where}\; r_Z(z_{ib}z_{ja})=z_{i a^{\prime}}z_{jb^{\prime}}\; \text{and}\; b > a^{\prime}.
\end{array}
\end{equation}
Every relation $\varphi_{ij,ba}$ has leading monomial  $\LM (\varphi_{ij,ba})=  z_{ib}z_{ja}$.

There is a one-to one correspondence between the set of relations $\Re(\mathbb{A}_Z)$ and the set of nontrivial $r_{Z}$-orbits in
$Z\times Z.$
\end{rmk}

  Note that $A\circ B$ is a subalgebra of $A\otimes B$, so if $f=g$ holds in $A\circ B$, then it holds in $A\otimes B.$
\begin{lem}
    \label{lem:segremap}
In notation as above. Let  $(X, r_1)$ and $(Y, r_2)$  be solutions on the finite disjoint sets $X = \{x_1, \cdots, x_m\}$, and $Y = \{y_1, \cdots, y_n\}$, and
let $A =\cA(\textbf{k},  X, r_1)$, and  $B =\cA(\textbf{k},  Y, r_2)$ be the corresponding YB algebras.
 Let $(Z, r_Z)$ be the solution of order $mn$ from Definition-Notation \ref{def:Z}, and let $\mathbb{A}_Z= \cA(\textbf{k},  Z, r_Z)$ be its  YB-algebra.
 Then the assignment
 \[z_{11} \mapsto x_1\otimes y_1,  \; z_{12} \mapsto x_1\otimes y_2,\; \cdots, \; z_{mn} \mapsto x_m\otimes y_n \]
extents to an algebra homomorpgism $s_{m, n} : \mathbb{A}_Z \longrightarrow A\otimes_{\textbf{k}} B$.
%The image of the map $s_{m,n}$ is the Segre product $A\circ B=\bigoplus_{i \geq 0}(A_i\otimes B_i)$.
    \end{lem}
    \begin{proof}
Naturally we set
$s_{m,n} (z_{i_1a_1}\cdots z_{i_pa_p}): =(x_{i_1}\circ y_{a_1})\cdots   (x_{i_p}\circ y_{a_p})$, for all
words $z_{i_1a_1}\cdots z_{i_pa_p}\in \langle Z\rangle$ and then extend this
map linearly.
    Note that for each polynomial $\varphi_{ji, ba}\in \Re(\mathbb{A}_Z)$  given in (\ref{eq:viprelationsZ1})
one has
    \[s_{n,d}(\varphi_{ji, ba}) = f_{ji,ba}\in \Re_{a},\]
where the set $\Re_{a}$ is a  part of the relations of $A\circ B$, given in Theorem \ref{thm:rel_segre}.

We have shown that $f_{ji,ba}$ equals identically zero in $A\circ B=\bigoplus_{i \geq 0}  A_i\otimes_{\textbf{k}} B_i$, which is a subalgebra of $A\otimes B$
    and therefore $s_{n,d}(\varphi_{ji, ba}) =f_{ji,ba}= 0$ in $A\otimes B$.

Similarly  for each  $\varphi_{ij, ba}$ given in  (\ref{eq:viprelationsZ2}) one has
\[s_{n,d}(\varphi_{ij, ba}) = f_{ij,ba}\in \Re_{a},\]
thus
\[
s_{n,d}(\varphi_{ij, ba}) =0\; \text{holds in}\; A\circ B \subset A\otimes B.
\]
We have shown
 that the map $s_{m,n}$ agrees
   with the relations of the algebra $\mathbb{A}_Z$.
It follows that the map $s_{m, n} : \mathbb{A}_Z \longrightarrow A\otimes_{\textbf{k}} B$
is a well-defined homomorphism of algebras.
\end{proof}

\begin{dfn}
\label{def:segremap}
We call the map $s_{m, n} : \mathbb{A}_Z \longrightarrow A\otimes_{\textbf{k}} B$ from Lemma \ref{lem:segremap} \emph{the $(m,n)$-Segre map}.
\end{dfn}

\begin{thm}
\label{thm:segremap}
In notation as above.
Let  $(X, r_1)$ and $(Y, r_2)$  be solutions on the finite disjoint sets $X = \{x_1, \cdots, x_m\}$, and $Y = \{y_1, \cdots, y_n\}$, let $N=mn$, and let $A$, and  $B$ be the corresponding YB algebras.
 Let  $(Z,r_Z)$ be the solution on the set $Z = \{z_{11}, \cdots, z_{mn}\}$ defined in Definition-Notation \ref{def:Z}, and let  $\mathbb{A}_Z= \cA(\textbf{k},  Z, r_Z)$ be its YB-algebra.
Let $s_{m, n} : \mathbb{A}_Z \longrightarrow A\otimes_{\textbf{k}} B$ be the Segre map
 extending
 the assignment
 \[z_{11} \mapsto x_1\circ y_1,  \; z_{12} \mapsto x_1\circ y_2, \;\cdots, \; z_{mn} \mapsto x_m\circ y_n. \]
 \begin{enumerate}
 \item The image of the Segre map $s_{m, n}$ is the Segre product $A\circ B$. Moreover,
$s_{m, n} : \mathbb{A}_Z \longrightarrow A\circ B$  is a homomorphism of graded algebras.

 \item The kernel $\mathfrak{K}=\ker (s_{m,n})$ of the Segre map is generated by the set $\Re_s$  of $\binom{m}{2}\binom{n}{2}$
  linearly independent quadratic
binomials described below
\begin{equation}
  \label{eq:gammaij}
  \begin{array}{lll}
\Re_s= &\{\gamma_{ij,ba} = z_{ib} z_{ja} -   z_{ia^{\prime}}z_{jb^{\prime}},& 1\leq i,j\leq m, 1 \leq a,b\leq n \mid  \\
            && r_1( x_ix_j) > x_ix_j , \;\text{and}\; \; r_2(y_by_a) = y_{a^{\prime}}y_{b^{\prime}}
\;  \text{with}\;\;b >a^{\prime}\}.
\end{array}
\end{equation}
  \end{enumerate}
\end{thm}
\begin{proof}
(1)   Clearly, the image of the Segre map $s_{m, n}$ is the subalgebra of $A\otimes B$ generated by the set
$s_{m,n} (Z) = X\circ Y,$ where
\[
W = X\circ Y=  \{w_{11}= x_1\circ y_1< w_{12}= x_1\circ y_2 <\cdots <w_{mn}= x_m\circ y_n\}.
\]
By Theorem \ref{thm:rel_segre} the set  $W = X\circ Y$ generates the algebra $A\circ B$, hence
$s_{m,n}(\mathbb{A}_Z) = A\circ B$.

(2)  Note first that $\Re_s \subset \mathfrak{K}=\ker (s_{m,n})$. Indeed, it follows from the description given in (\ref{eq:gammaij}) that
if $\gamma_{ij,ba} = z_{ib} z_{ja} -   z_{ia^{\prime}}z_{jb^{\prime}}\in \Re_s$ then the element $\gamma_{ij,ba}$ is different from zero
in $\mathbb{A}_Z$ and direct computation shows that
\[s_{m,n}(\gamma_{ij,ba})= s_{m,n}( z_{ib} z_{ja} -   z_{ia^{\prime}}z_{jb^{\prime}}) =  (x_i\circ y_b)(x_j\circ y_a)-(x_i\circ y_{a^{\prime}})(x_j\circ y_{b^{\prime}})= g_{ij,ba} \in \Re_b,\]
where $\Re_b$ is the subset of relations  of $A\circ B$ described in  (\ref{eq:viprelations3}). Therefore, $s_{m,n}(\gamma_{ij,ba})= 0$ in $A\circ B,$
for all elements of $\Re_s$, and
\[
\Re_s \subset \mathfrak{K}=\ker (s_{m,n}).\]
 Moreover, there is a one-to-one correspondence between the sets $\Re_s$ and $\Re_b,$ so
\begin{equation}
  \label{eq:Res}
|\Re_s|= |\Re_b|= \binom{m}{2}\binom{n}{2}.
\end{equation}

We claim that $\Re_s$ is a  minimal set of generators of the kernel $\mathfrak{K}$.
The set $\Re_s$ is linearly independent, since $s_{m,n}(\Re_s)= \Re_b$, and $\Re_b$ is a linearly independent set in $A\otimes B$, by Lemma
\ref{lem:independence}.

By the First Isomorphism Theorem $\mathbb{A}_Z/\mathfrak{K} \simeq A\circ B$ and
$(\mathbb{A}_Z/\mathfrak{K})_2 \simeq (A\circ B)_2.$
Hence
\[
\dim (\mathbb{A}_Z)_2= \dim(\mathfrak{K}_2)+\dim(A\circ B)_2,
\]
and
\[ \dim (\mathfrak{K}_2)= \dim(\mathbb{A}_Z)_2-\dim (A\circ B)_2. \]
But $\mathbb{A}_Z$ is the YB-algebra of the solution $(Z, r_Z)$ of order $mn$, so
 $\dim(\mathbb{A}_Z)_2= \binom{mn+1}{2}$. We also know that $\dim (A\circ B)_2= \dim A_2\dim B_2 = \binom{m+1}{2} \binom{n+1}{2}$.
It follows that
\begin{equation}
  \label{eq:dimK_2}
\dim (\mathfrak{K})_2 = \binom{mn+1}{2}-\binom{m+1}{2} \binom{n+1}{2} =\binom{m}{2}\binom{n}{2},
\end{equation}
Therefore by ( \ref{eq:Res})
\[ \dim (\mathfrak{K})_2 =|\Re_s|,\]
and since $\Re_s$ is linearly independent subset of  $\mathfrak{K}_2$ it follows that $\Re_s$ is a basis of the graded component $\mathfrak{K}_2$ of
the ideal $\mathfrak{K}$.
Therefore $\mathfrak{K}_2 = \textbf{k}\Re_s$. The ideal $\mathfrak{K}$ is generated by homogeneous polynomials of degree $2$,
hence
\[\mathfrak{K} =(\mathfrak{K}_2)= (\Re_s).\]
We have shown that $\Re_s$ is a minimal set of generators for the kernel $\mathfrak{K}$.
\end{proof}

\begin{cor}
\label{cor:SegreProdNoetherian}
In notation and assumption as above. Let $A=\cA(\textbf{k}, X,r_1),$ and $B=\cA(\textbf{k}, Y,r_2),$ be the YB algebras of the finite solutions $(X,r_1)$
and $(Y, r_2)$ of order $m$ and $n$, respectively.
Then the Segre product $A\circ B$ is a left and right Noetherian algebra. Moreover, $A\circ B$ has polynomial growth.
\end{cor}
\begin{proof} It follows from Theorem  \ref{thm:segremap}  that $A\circ B= s_{m, n}  (\mathbb{A}_Z)$,  the image of the Segre homomorphism
$s_{m, n} : \mathbb{A}_Z \longrightarrow A\otimes_{\textbf{k}} B$ , where $\mathbb{A}_Z$ is the YB-algebra of the solution $(Z, r_Z)$ of order $mn$.
 By \cite[Theorem 1.4]{GIVB} the algebra $\mathbb{A}_Z$  is left and right Noetherian, and therefore its homomorphic image $A\circ B$ is left and right Noetherian.  Moreover, $\mathbb{A}_Z$ has polynomial growth of degree $mn$ and therefore $A\circ B$ has polynomial growth.
\end{proof}

\begin{que}
\label{que} Let $A=\cA(\textbf{k}(X,r_1),$  and $B=\cA(\textbf{k}(Y,r_2),$ be the YB algebras of the finite solutions $(X,r_1)$
and $(Y, r_2)$. Is it true that the Segre product $A \circ B$ is a domain?
\end{que}
A modified v1ersion for the the particular cases of square-free solutions is still open.
\begin{que}
\label{que1} Let $A=\cA(\textbf{k}(X,r_1),$  and $B=\cA(\textbf{k}(Y,r_2),$ be the YB algebras of the finite square-free solutions $(X,r_1)$
and $(Y, r_2)$. Is it true that the Segre product $A \circ B$ is a domain?
\end{que}
We expect that due to the very special properties of $A$ and $B$, and the specific relations of $A \circ B$ the Segre product $A \circ B$ is also a domain.

We know that $A$ and $B$  are Noetherian domains, and $A\circ B$ is a subalgebra of the tensor product $A\otimes B$. However, it is shown in \cite{rowen} that the tensor product $D_1\otimes_F D_2$ of two division algebras over an algebraically closed field contained in their centers may not be a domain.

\section{Segre products and Segre maps for the class of square-free solutions}
\label{sec:square-free}
In this section $(X,r_1)$ and $(Y, r_2)$ are fixed disjoint square-free solutions of orders $|X|= m$ and $|Y|= n.$
Let
 $A =\cA(\textbf{k},  X, r_1)$, and  $B =\cA(\textbf{k},  Y, r_2)$  be the corresponding YB-algebras.
As in Convention \ref{rmk:conventionpreliminary1} we enumerate
$X = \{x_1, \cdots, x_m\}$, and $Y = \{y_1, \cdots, y_n\}$,
so that $A$ and $B$ are binomial skew-polynomial algebras with respect to these enumerations, see Definition \ref{binomialringdef}. In particular, $A$ is a PBW algebra with PBW generators $x_1, \cdots, x_m$ and $B$ is a PBW algebra with PBW generators $y_1, \cdots, y_n$.
The following result collects various algebraic properties of the Segre product $A\circ B.$

\begin{thm}
\label{thm:segre_square-free}
Suppose $(X,r_1)$ and $(Y, r_2)$ are disjoint square-free solutions, where
\[X = \{x_1, \cdots, x_m\}, \quad \text{and} \quad Y = \{y_1, \cdots, y_n\}
\]
are enumerated so that the YB-algebras $A =\cA(\textbf{k},  X, r_1)$, and  $B =\cA(\textbf{k},  Y, r_2)$  are binomial skew polynomial rings with respect to these enumerations.

 Then the Segre product $A\circ B$ satisfies the following conditions.
 \begin{enumerate}
 \item
 \label{thm:segre_square-free1}
 $A\circ B$ is a PBW algebra with a set of $mn$ PBW generators
\[W=X\circ Y =\{w_{11}=x_1\circ y_1, \; w_{12}= x_1\circ x_2,\; \cdots,\; w_{1n}=x_1\circ y_n,\; \cdots,\; w_{mn}=x_m\circ x_n \},\]
ordered lexicographically.
and a standard finite presentation
\[
A\circ B \simeq \textbf{k} \langle w_{11}, \cdots, w_{mn}\rangle/(\Re) ,
\]
where the set of relations $\Re$ is a Gr\"{o}bner basis
of the ideal $I = (\Re)$ and consists of $\binom{mn}{2}+\binom{m}{2}\binom{n}{2}$  square-free quadratic polynomials. The set $\Re$ splits as a disjoint union
$\Re= \Re_a \cup \Re_b$ of the sets $\Re_a$ and $\Re_b$ described below.
\begin{enumerate}
\item

The set $\Re_a$ consists of $\binom{mn}{2}$ relations given explicitly in (\ref{eq:viprelations1a6})  and  (\ref{eq:viprelations2a6}) :
\begin{equation}
\label{eq:viprelations1a6}
\begin{array}{c}
f_{ji,ba}= w_{jb}w_{ia}-w_{i^{\prime}a^{\prime}}w_{j^{\prime}b^{\prime}}, 1\leq i<j \leq m, 1 \leq a, b \leq n,
\; \text{where} \\
r_1(x_jx_i)  = x_{i^{\prime}}x_{j^{\prime}}, \;j > i^{\prime}\; \text{and} \;r_2(y_by_a)=y_{a^{\prime}}y_{b^{\prime}};\\
w_{jb}>w_{i^{\prime}a^{\prime}}, \;\; w_{jb}>w_{ia}, \;\; w_{i^{\prime}a^{\prime}}< w_{j^{\prime}b^{\prime}}.
\end{array}
\end{equation}
Every relation $f_{ji,ba}$ has leading monomial  $\LM (f_{ji,ba})= w_{jb}w_{ia}$.

\begin{equation}
\label{eq:viprelations2a6}
\begin{array}{c}
f_{ii,ba}= w_{ib}w_{ia}- w_{ia^{\prime}} w_{ib^{\prime}}, 1\leq i \leq m, 1 \leq a<b\leq n,\;  \text{where}\\
r_2(y_by_a) = y_{a^{\prime}}y_{b^{\prime}}\; \text{with}\;\; b > a^{\prime}\\
w_{ib}> w_{ia^{\prime}},\;\; w_{ib}> w_{ia},\;\; w_{ia^{\prime}}< w_{ib^{\prime}}.
\end{array}
\end{equation}
Every relation $f_{ii,ba}$ has leading monomial  $\LM (f_{ii,ba})= w_{ib}w_{ia}$.

\item
The set $\Re_b$ consists of $\binom{m}{2}\binom{n}{2}$ relations given explicitly in (\ref{eq:viprelations3b6})
\begin{equation}
\label{eq:viprelations3b6}
\begin{array}{c}
g_{ij, ba}= w_{ib}w_{ja}- w_{ia^{\prime}} w_{jb^{\prime}},   1\leq i<j\leq m, \; 1 \leq a<b\leq n, \\
\text{where}\; r_2(y_by_a) = y_{a^{\prime}}y_{b^{\prime}}\;\;
 \text{with}\;\;b >a^{\prime}.
\end{array}
\end{equation}
One has $\LM (g_{ij,ba})= w_{ib}w_{ja}$.
\end{enumerate}
\item
\label{thm:segre_square-free2}
$A\circ B$ is a Koszul algebra.
\item
\label{thm:segre_square-free3}
$A\circ B$ is left and right Noetherian.
\item
\label{thm:segre_square-free4}
The algebra $A\circ B$ has polynomial growth and infinite global dimension.
\end{enumerate}
\end{thm}
\begin{proof}
 (\ref{thm:segre_square-free1}). It follows from \cite[Chap 4.4, Proposition 4.2]{PoPo} that the Segre product $A\circ B$ is a PBW algebra with a set of PBW one-generators the set $W=X\circ Y$, ordered lexicographically. The shape of the defining relations follows from Theorem   \ref{thm:rel_segre},  and from the relations of the binomial skew-polynomial rings $A$ and $B$ encoding  the properties of $r_1$ and $r_2$. To show that $\Re$ is a Gr\"{o}bner basis of the ideal $I =(\Re)$
 it will be enough to check that
 \[\Ncal(\Re)_3= \Ncal(I)_3. \]
In general,  $ \Ncal(I)_3 \subseteq \Ncal(\Re)_3$, so we have to show that
 \[|\Ncal(\Re)_3|=|\Ncal(I)_3|=\dim (A\circ B)_3 = \binom{m+2}{3}\binom{n+2}{3}. \]
A monomial $u \in \langle W \rangle$ of length $3$ is normal modulo $\Re$ \emph{iff} it is normal modulo $\overline{\Re}$, that is
$u$ is not divisible by any of the leading monomials $\LM(f), f\in \Re$.
Note that
\[w_{ia}w_{jb}w_{kc}\in \Ncal(\Re)_3 \Longleftrightarrow w_{ia}w_{jb}\in \Ncal(\Re)_2 \;\text{and}\;w_{jb}w_{kc}\in \Ncal(\Re)_2 \]
It follows from the shape of the leading monomials $\LM(f), f\in \Re$ that  $w_{ia}w_{jb}\in \Ncal(\Re)_2$ if and only if
$1 \leq i\leq j\leq m$ and $1 \leq a\leq b\leq n.$
Therefore
\[w_{ia}w_{jb}w_{kc}\in \Ncal(\Re)_3 \Longleftrightarrow  1\leq i \leq j \leq k \leq m\; \text{and}\;1\leq a\leq b \leq c\leq n.\]
In other words
\[\Ncal(\Re)_3 =\{w_{ia}w_{jb}w_{kc} \mid 1 \leq i \leq j \leq k\leq m, \; 1 \leq a \leq b \leq k\leq m \}.\]
It follows that $|\Ncal(\Re)_3| =  \binom{m+2}{3}\binom{n+2}{3}= \dim (A\circ B)_3,$ as desired.
Therefore the set of defining relations $\Re $ is a Gr\"{o}bner basis.

 (\ref{thm:segre_square-free2}).
 The Kosulity of $A\circ B$ follows from Corollary
\ref{cor:SegreProductProperties}. But it follows also from the fact that every PBW algebra is Koszul, see \cite{priddy}.

 (\ref{thm:segre_square-free3}).  Corollary \ref{cor:SegreProdNoetherian} implies that $A\circ B$ is left and right Noetherian.

 (\ref{thm:segre_square-free4}).  By Corollary \ref{cor:SegreProdNoetherian} again, $A\circ B$ has polynomial growth.
It follows from  Theorem 1.1, \cite{GI12} that if
 a graded PBW algebra has $mn$ one-generators, polynomial growth and finite global dimension, then the number of its defining relation must be $\binom{mn}{2}$.  We have shown in part (1) that the algebra $A\circ B$ is a quadratic PBW algebra with $mn$ PBW generators, and $\binom{mn}{2}+\binom{m}{2}\binom{n}{2}$ defining relations, therefore  $A\circ B$ has infinite global dimension.
\end{proof}
As we mentioned before we do not know if $A\circ B$ is a domain even in the square-free case, see
Question \ref{que1}.

Let $(X\circ Y, r_{X\circ Y})$ be the solution on the set $X\circ Y$,
 defined
in Proposition-Notation \ref{pronotation}. Then $(X\circ Y, r_{X\circ Y})$ is a square-free solution.

Consider now the solution $(Z, r_Z)$ on the set $Z=\{z_{11}, z_{12}, \cdots, z_{mn}\}$ defined in Definition-Notation \ref{def:Z}. By construction $(Z, r_Z)$ is isomorphic to the solution $(X\circ Y, r_{X\circ Y})$, and therefore it is a square-free solution of order $mn$.

\begin{pro}
\label{pro:SegreSquarefree}
Let  $\mathbb{A}_Z= \cA(\textbf{k},  Z, r_Z)$ be the  YB-algebra of  $(Z, r_Z)$.
Then $\mathbb{A}_Z$ is a binomial skew-polynomial ring
with a standard finite presentation
$\mathbb{A}_Z =\textbf{k}\langle z_{11}, z_{12}, \cdots, z_{mn}\rangle /(\Re(\mathbb{A}_Z))$, where
the set of defining relations $\Re(\mathbb{A}_Z)$ consists
of
$\binom{mn}{2}$ binomial relations given explicitly in (\ref{eq:viprelationsZ1a6})  and  (\ref{eq:viprelationsZ26}).
\begin{equation}
\label{eq:viprelationsZ1a6}
\begin{array}{c}
\varphi_{ji,ba}= z_{jb}z_{ia}- z_{i^{\prime}a^{\prime}}z_{j^{\prime}b^{\prime}},
1 \leq i<j\leq m, 1 \leq a,b\leq n, \; \text{where} \\
r_Z(z_{jb}z_{ia}) = z_{i^{\prime}a^{\prime}}z_{j^{\prime}b^{\prime}}, \;\text{and}\;
z_{jb}>z_{ia},\; z_{i^{\prime}a^{\prime}}< z_{j^{\prime}b^{\prime}},\; z_{jb}> z_{i^{\prime}a^{\prime}}.
\end{array}
\end{equation}
Every relation $\varphi_{ji,ba}$ has leading monomial  $\LM (\varphi_{ji,ba})=z_{jb}z_{ia}$.

\begin{equation}
\label{eq:viprelationsZ26}
\begin{array}{c}
\varphi_{ii,ba}= z_{ib}z_{ia}- z_{i a^{\prime}}z_{ib^{\prime}}, \;
1 \leq i\leq m, 1 \leq a < b\leq n, \; \text{where}\\
 r_Z(z_{ib}z_{ia})=z_{i a^{\prime}}z_{ib^{\prime}}\; \text{and}\;
 z_{ib}> z_{ia}, \;z_{i a^{\prime}}< z_{ib^{\prime}}, \; z_{ib}>z_{i a^{\prime}}.
\end{array}
\end{equation}
Every relation $\varphi_{ij,ba}$ has leading monomial  $\LM (\varphi_{ii,ba})=  z_{ib}z_{ia}$.
Moreover, the set $\Re(\mathbb{A}_Z)$ forms a Gr\"{o}bner basis of the ideal $I=(\Re(\mathbb{A}_Z))$ of $\textbf{k}\langle z_{11}, z_{12}, \cdots, z_{mn}\rangle$
with respect to the degree-lexicographic order.
\end{pro}
\begin{proof}
The relations $\Re(\mathbb{A}_Z)$ described with details in (\ref{eq:viprelationsZ1a6}) and (\ref{eq:viprelationsZ26})
have the shape of the typical relations of a binomial skew-polynomial ring, see Definition \ref{binomialringdef}, conditions (a), (b), (c).
We have to show that the set $\Re(\mathbb{A}_Z)$ is a Gr\"{o}bner basis of the ideal $I$ with respect to the degree-lexicographic order on $\langle z_{11}, z_{12}, \cdots, z_{mn}\rangle$.
It follows from the shape of relations that the set $\cN(I)$ of normal words modulo $I$ is a subset of the set of terms (ordered monomials)
\[
\cN(I) \subseteq \cT(Z)= \{z_{11}^{k_{11}}\cdots z_{mn}^{k_{mn}}\mid k_{ia}\geq 0,\; 1 \leq i \leq m,\; 1 \leq a \leq n  \}.
\]
By Facts \ref{fact2}  the Yang-Baxter algebra $\mathbb{A}_Z$ has Hilbert series $H_{\mathbb{A}} (t) = \frac{1}{(1-t)^{mn}}$ which implies that $\cN(I) = \cT(Z)$. In other words $\cT(Z)$ is the normal $\textbf{k}$-basis of
 $\mathbb{A}_Z$, so condition (d$^{\prime}$) in Definition \ref{binomialringdef} is satisfied, and therefore $\Re(\mathbb{A}_Z)$ is a Gr\"{o}bner basis of the ideal $I$.
 \end{proof}

\begin{cor}
\label{cor:segresqfree}
In notation as above. Let  $(X, r_1)$ and $(Y, r_2)$  be disjoint square-free solutions of finite orders, $m$ and $n$, respectively,
let $A $, and  $B $ be the corresponding YB-algebras.
Let  $(Z,r_Z)$ be the square-free solution on the set $Z = \{z_{11}, \cdots, z_{mn}\}$ defined in Definition-Notation \ref{def:Z}, and let  $\mathbb{A}= \cA(\textbf{k},  Z, r_Z)$ be its YB-algebra.
Let $s_{m, n} : \mathbb{A} \longrightarrow A\otimes_{\textbf{k}} B$ be the Segre map
 extending
 the assignment
 \[z_{11} \mapsto x_1\circ y_1, \;\; z_{12} \mapsto x_1\circ y_2,\;\;  \cdots, \;\; z_{mn} \mapsto x_m\circ y_n. \]
 \begin{enumerate}
 \item The image of the Segre map $s_{m, n}$ is the Segre product $A\circ B$.
 \item The kernel $\mathfrak{K}=\ker (s_{m,n})$ of the Segre map is generated by the set $\Re_s$  of $\binom{m}{2}\binom{n}{2}$
  linearly independent quadratic
binomials described below
\begin{equation}
  \label{eq:gammaij6}
  \begin{array}{lll}
\Re_s= &\{\gamma_{ij,ba} = z_{ib} z_{ja} -   z_{ia^{\prime}}z_{jb^{\prime}} \mid & 1\leq i<j\leq m, 1 \leq a< b\leq n \\
            &&  \;\text{and}\; \; r_2(y_by_a) = y_{a^{\prime}}y_{b^{\prime}}
\;  \text{with}\;\;b >a^{\prime}, a^{\prime} < b^{\prime}\}.
\end{array}
\end{equation}
  \end{enumerate}
\end{cor}

\section{An Example}
\label{sec:examples}
We shall present an example which illustrates the results of the paper. We use
the notation of the previous sections.

\begin{ex}[]\label{ex:n=3mod}
Let  $(X, r_1)$, be the square-free solution on the set $X= \{x_1, x_2, x_3\}$, where
\[
\begin{array}{llll}
r_1: &x_3x_2 \leftrightarrow  x_1x_3 & x_3x_1 \leftrightarrow
x_2x_3 & x_2x_1 \leftrightarrow  x_1x_2\\
     & x_ix_i \leftrightarrow x_ix_i&1 \leq i \leq 3. &\\
 \end{array}
\]
and let $(Y, r_2)$ be the solution on $Y=\{y_1, y_2\}$, where
\[
\begin{array}{llll}
r_2: &y_2y_2 \leftrightarrow y_1y_1& y_1y_2 \leftrightarrow y_1y_2 &  y_2y_1 \leftrightarrow y_2y_1.
      \end{array}
\]
Then
 \[
\begin{array}{l}
A=\cA(\textbf{k},X,r_1) = \textbf{k}\langle x_1, x_2, x_3  \rangle /(x_3x_2 -x_1x_3,\;  x_3x_1 -x_2x_3,\;   x_2x_1 -x_1x_2);\\
B=\cA(\textbf{k},Y,r_2) = \textbf{k}\langle y_1, y_2  \rangle /(y_2^2-y_1^2).
 \end{array}
\]
The algebra $A$ is a binomial skew-polynomial ring. $B$ is not a PBW algebra, see more details in \cite{GI22}.

Let $A\circ B$ be the Segre product of $A$ and $B$, and let $(X\circ Y, r_{X\circ Y})$ be the solution from Proposition-Notation \ref{pronotation} isomorphic to the Cartesian product of solutions $(X\times Y, \rho_{X\times Y})$.
Then $A\circ B$ is a quadratic algebra with a set of one-generators
\[W =\{ w_{11}=x_1\circ y_1, w_{12}=x_1\circ y_2, w_{21}=x_2\circ y_1, w_{12}=x_2\circ y_2, w_{31}=x_3\circ y_1, w_{32}=x_3\circ y_2\}\]
and $18$ defining quadratic relations. More precisely,
\[
A\circ B \simeq \textbf{k} \langle w_{11}, w_{12}, w_{21}, w_{22}, w_{31}, w_{32}\rangle/(\Re) ,
\]
where
$\Re= \Re_a \cup \Re_b$ is a disjoint union of quadratic relations $\Re_a$ and $\Re_b$ given below.
\begin{enumerate}
\item The set $\Re_a$ with $|\Re_a|=15$ is a disjoint union  $\Re_a= \Re_{a1} \cup
\Re_{a2}$, where
\[
\begin{array}{lll}
\Re_{a1}= \{&f_{32,22}=  w_{32}w_{22}-w_{11}w_{31},\quad f_{32,11}=  w_{31}w_{21}-w_{12}w_{32},&\\
            &f_{32,21}=  w_{32}w_{21}-w_{12}w_{31},\quad f_{32,12}=  w_{31}w_{22}-w_{11}w_{32}, &\\
            &f_{31,22}=  w_{32}w_{12}-w_{21}w_{31},\quad f_{31,11}=  w_{31}w_{11}-w_{22}w_{32}, &\\
            &f_{31,21}=  w_{32}w_{11}-w_{22}w_{31},\quad f_{31,12}=  w_{31}w_{12}-w_{21}w_{32}, &\\
            &f_{21,22}=  w_{22}w_{12}-w_{11}w_{21},\quad f_{21,11}=  w_{21}w_{11}-w_{12}w_{22}, &\\
            &f_{21,21}=  w_{22}w_{11}-w_{12}w_{21}, \quad f_{21,12}=  w_{21}w_{12}-w_{11}w_{22}& \}.\\
            &&\\
\Re_{a2}= \{& f_{33,22}=  w_{3 2}w_{3 2}-w_{31}w_{31},\quad f_{22,22}=  w_{22}w_{22}-w_{21}w_{21},& \\
            &f_{11,22}=  w_{12}w_{12}-w_{11}w_{11} \; \}.&
            \end{array}
\]
In fact the relations $\Re_a$  are exactly the defining relations of the YB-algebra $\Alg_{X\circ Y}= \cA(\textbf{k}, X\circ Y, r_{X\circ Y})$,
there is a one-to one correspondence between the set of relations $\Re_a$ and the set of nontrivial $r_{X\circ Y}$-orbits in
$(X\circ Y) \times (X\circ Y)$.
\item The set $\Re_b$ consists of $3$ quadratic relations given below
\[
%\begin{array}{ll}
\Re_{b}= \{g_{23,22}= w_{22}w_{32}-w_{21}w_{31},\;  g_{13,22}= w_{12}w_{32}-w_{11}w_{31},\;g_{12,22}= w_{12}w_{22}-w_{11}w_{21}\}.
\]
\end{enumerate}
Let
$Z=\{z_{11}, z_{12}, z_{21}, z_{22}, z_{31},z_{11},z_{31}\}$, and let $(Z, r_Z)$ be the solution
 defined in Definition-Notation \ref{def:Z}. By construction $(Z, r_Z)$ is isomorphic to the solution $(X\circ Y, r_{X\circ Y})$.

The Yang-Baxter algebra $\mathbb{A}_Z= \cA(\textbf{k},  Z, r_Z)$
has a finite  presentation
\[\mathbb{A}_Z = \textbf{k} \langle z_{11}, z_{12}, z_{21}, z_{22}, z_{31}, z_{32}\rangle/(\Re(\mathbb{A}_Z)),\]
where the set $\Re(\mathbb{A}_Z))$ of 15 defining relations is:
\[
\begin{array}{lll}
\Re(\mathbb{A}_Z)= \{&\varphi_{32,22}=  z_{32}z_{22}-z_{11}z_{31},\quad \varphi_{32,11}=  z_{31}z_{21}-z_{12}z_{32},&\\
            &\varphi_{32,21}=  z_{32}z_{21}-z_{12}z_{31},\quad \varphi_{32,12}=  z_{31}z_{22}-z_{11}z_{32}, &\\
            &\varphi_{31,22}=  z_{32}z_{12}-z_{21}z_{31},\quad \varphi_{31,11}=  z_{31}z_{11}-z_{22}z_{32}, &\\
            &\varphi_{31,21}=  z_{32}z_{11}-z_{22}z_{31},\quad \varphi_{31,12}=  z_{31}z_{12}-z_{21}z_{32}, &\\
            &\varphi_{21,22}=  z_{22}z_{12}-z_{11}z_{21},\quad \varphi_{21,11}=  z_{21}z_{11}-z_{12}z_{22}, &\\
            &\varphi_{21,21}=  z_{22}z_{11}-z_{12}z_{21}, \quad \varphi_{21,12}=  z_{21}z_{12}-z_{11}z_{22}, &\\
            & \varphi_{33,22}=  z_{3 2}z_{3 2}-z_{31}z_{31},\quad \varphi_{22,22}=  z_{22}z_{22}-z_{21}z_{21},& \\
            &\varphi_{11,22}=  z_{12}z_{12}-z_{11}z_{11} \; \}.&
            \end{array}
\]
The Segre map
$s_{3,2}:\; \mathbb{A}_Z \longrightarrow A\otimes B$
has image $A\circ B$. The kernel $\ker(s_{3,2})$
is the ideal of $\mathbb{A}_Z$ generated by the three polynomials
\[
\gamma_{23,22}= z_{22}z_{32}-z_{21}z_{31},\;  \gamma_{13,22}= z_{12}z_{32}-z_{11}z_{31},\;\gamma_{12,22}= z_{12}z_{22}-z_{11}z_{21}.
\]

 \end{ex}

\end{document}